\documentclass[12pt]{article}
\usepackage{amsmath}
\usepackage{amsthm}
\usepackage{amssymb}
\usepackage{cite}
\usepackage{url}
\usepackage[multiple]{footmisc}
\usepackage{graphicx}
\usepackage{epstopdf}
\usepackage{chngcntr}
\usepackage{enumitem}
\usepackage{hhline}
\usepackage{array}
\usepackage{hyperref}
\usepackage{enumitem}
\usepackage{color}
\usepackage{subcaption}
\usepackage{float}
\usepackage[ruled]{algorithm2e}

\hypersetup{colorlinks=false}

\if@twoside
\else 
\fi 

\numberwithin{equation}{section}
\counterwithin{table}{section}

\theoremstyle{plain}% default
\newtheorem{theorem}{Theorem}[section]

\newtheorem{proposition}[theorem]{Proposition}

\theoremstyle{definition}

\newtheorem{example}{Example}[section]

\newtheorem{assumption}{Assumption}[section]

\theoremstyle{remark}
\newtheorem*{remark}{Remark}

\newcommand{\norm}[1]{\left\|#1\right\|}
\newcommand{\abs}[1]{\left\vert#1\right\vert}
\newcommand{\spr}[1]{\left\langle\,#1\,\right\rangle}
\newcommand{\sspr}[1]{\left\langle \,#1\,\right\rangle_{s,\beta}}
\newcommand{\kl}[1]{\left(#1\right)}
\newcommand{\Kl}[1]{\left\{#1\right\}}
\newcommand{\vl}{\, \vert \,}

\newcommand{\intV}[1]{\int\limits_\Omega #1 \, dx}

\newcommand{\intS}[1]{\int\limits_{\partial\Omega} #1 \, dS}

\newcommand{\R}{\mathbb{R}} 

\newcommand{\N}{\mathbb{N}}

\newcommand{\grad}{\nabla}

\renewcommand{\div}[1]{\operatorname{div}\left(#1\right)}

\newcommand{\Ht}{{H^2}}

\newcommand{\HsO}{{H^s(\Omega)}}

\newcommand{\HoO}{{H^1(\Omega)}}

\newcommand{\HtO}{{H^2(\Omega)}}

\newcommand{\Lt}{{L^2}}

\newcommand{\LoO}{{L^1(\Omega)}}
\newcommand{\LtO}{{L^2(\Omega)}}

\newcommand{\LiO}{{L^\infty(\Omega)}}

\newcommand{\xD}{x^\dagger}

\newcommand{\xkd}{x_k^\delta}

\newcommand{\xkpd}{x_{k+1}^\delta}

\newcommand{\xksd}{x_{k_*}^\delta}

\newcommand{\skd}{s_k^\delta}

\newcommand{\okd}{{\omega_k^\delta}}

\newcommand{\yd}{y^{\delta}}

\newcommand{\cLM}{c_{LM}}

\newcommand{\ks}{{k_*}}

\newcommand{\bO}{{\partial \Omega}}

\newcommand{\eps}{\varepsilon}

\newcommand{\HdoO}{{H_\diamond^1(\Omega)}}

\newcommand{\n}{\vec{n}}

\newcommand{\Es}{E_{s}}
\newcommand{\Esb}{E_{s,\beta}}

\newcommand{\sib}{\underline{\sigma}}
\newcommand{\Ed}{E^\delta}

\newcommand{\sD}{{\sigma^\dagger}}

\newcommand{\Msb}{\mathcal{M}(\sib)}

\newcommand{\Ds}{\mathcal{D}_s}

\newcommand{\Htb}{H^2_\beta}

\title{Limited Angle Acousto-Electrical Tomography}
\author{
Simon Hubmer\footnote{Johannes Kepler University Linz, Doctoral Program Computational Mathematics, Altenbergerstra{\ss}e 69, A-4040 Linz, Austria (simon.hubmer@dk-compmath.jku.at)},
Kim Knudsen\footnote{Technical University of Denmark, Department of Applied Mathematics and Computer Science, Asmussens All\'e, 2800 Kongens Lyngby, Denmark (kiknu@dtu.dk)},
Changyou Li\footnote{Technical University of Denmark, Department of Applied Mathematics and Computer Science, Asmussens All\'e, 2800 Kongens Lyngby, Denmark (chgy@dtu.dk) and School of Electronics and Information, Northwestern Polytechnical University, Xian, China},
Ekaterina Sherina\footnote{Technical University of Denmark, Department of Applied Mathematics and Computer Science, Asmussens All\'e, 2800 Kongens Lyngby, Denmark (sershe@dtu.dk), corresponding author}
}

\begin{document}

\maketitle
\begin{abstract}
This paper considers the reconstruction problem in Acousto-Electrical Tomography, i.e., the problem of estimating a spatially varying conductivity in a bounded domain from measurements of the internal power densities resulting from different prescribed boundary  conditions. Particular emphasis is placed on the limited angle scenario, in which the boundary conditions are supported only on a part of the boundary.  The reconstruction problem is formulated as an optimization problem in a Hilbert space setting and solved using Landweber iteration. The resulting algorithm is implemented numerically in two spatial dimensions and tested on simulated data. The results quantify the intuition that features close to the measurement boundary are stably reconstructed and features further away are less well reconstructed. Finally, the ill-posedness of the limited angle problem is quantified numerically using the singular value decomposition of the corresponding linearized problem.


\medskip
\noindent \textbf{AMS:} 65J22, 35R30, 65M32 
% 47J06 - Operator Theory - Nonlinear Ill-Posed Problems
% 35R30 - PDE - Misc - Inverse Problems 
% 65M32 - Numana - PDE - Inverse Problems
\end{abstract}

% % % % % % % % %
% Introduction  %
% % % % % % % % %
\section{Introduction}

Electrical Impedance Tomography is an emerging technology that  aims at reconstructing the spatially varying electric conductivity distribution in a body from electrostatic measurements of voltages and the corresponding current fluxes on the surface of the body. The quantitative and structural information acquired about the conductivity of the body  can potentially be valuable for medical and industrial applications. For example, EIT shows great promise for bed side lung monitoring \cite{Holder_2005} and for non-destructive testing of concrete \cite{Karhunen_2010,Karhunen_2010a}.

The reconstruction problem in EIT is well-known for being (severely) ill-posed \cite{Mandache_2001}. To overcome the ill-posedness, a novel idea of coupling EIT with a different physical phenomenon has been promoted in the last decade. EIT used together with magnetic resonance leads to so-called Magnetic Resonance EIT \cite{Seo_2011a}, whereas EIT modulated by ultrasound waves leads to Acousto-Electrical Tomography \cite{ZhangWang_2004,Ammari_Bonnetier_Capdeboscq_Tanter_Fink_2008, Kuchment_Kunyansky_2010} (or equivalently Impedance-Acoustic Tomography (IAT) \cite{Gebauer_Scherzer_2008}). Both modalities give rise to additional interior information and may potentially lead to a significant improvement of the conductivity reconstructions having both high contrast and resolution. 

In this paper we focus on Acousto-Electrical Tomography (AET). Denote by $\sigma$ the spatially varying conductivity in the bounded and smooth domain $\Omega \subset \R^N$, $N=2,3$. The power density is defined as
	\begin{equation} \label{eq_powden}
		E(\sigma) := \sigma \abs{\grad u(\sigma)}^2 \,,
	\end{equation}
where $u(\sigma)$ denotes the interior voltage potential given as the solution of the elliptic equation
	\begin{equation}\label{eq_1}
        \div{\sigma \grad u} = 0 \,, \qquad \text{in } \Omega \,.
	\end{equation}
The goal is to reconstruct $\sigma$ from knowledge of $E$, where the data $E$ can be obtained through the AET procedure \cite{Bal_Bonnetier_Monard_Triki_2013} and, moreover, $E$ is connected to the conductivity $\sigma$ via \eqref{eq_powden}.
Most studies \cite{Bal_Naetar_Scherzer_Scotland_2013,Capdeboscq_Fehrenbach_Gournay_Kavian_2009,Ammari_Bonnetier_Capdeboscq_Tanter_Fink_2008} consider the case of \eqref{eq_1} being supplemented with Dirichlet conditions on the boundary $\partial \Omega$ 
	\begin{equation} \label{eq_dir}
	  	u\vert_\bO = f \,.
	\end{equation}
In contrast, this paper considers \eqref{eq_1} supplemented with Neumann boundary conditions
	\begin{equation} \label{eq_neu}
		(\sigma \grad u) \cdot \n \vert_\bO = g \,.
	\end{equation}

Note that physically the function $g$ measures the  current flux on the boundary in the normal direction given by the outward unit normal $\n$ to $\partial \Omega$. Neumann boundary conditions, which model the current flux along the boundary, are the natural boundary conditions for EIT, and they also form the basis of more sophisticated models for EIT like the complete electrode model \cite{Somersalo_Cheney_Isaacson_1992}. Since EIT forms the basis of AET, Neumann boundary conditions are also natural for AET \cite{Ammari_Bonnetier_Capdeboscq_Tanter_Fink_2008}.

The AET procedure makes use of perturbations in the conductivity caused by an ultrasound wave sent through the body \cite{Bal_Bonnetier_Monard_Triki_2013}. The wave (given by $p(x,t)$) perturbs the conductivity slightly into \cite{Jossinet_Lavandier_Cathignol_1998,Lavandier_Jossinet_Cathignol_2000}
  \begin{align*}
    \sigma_\varepsilon = \sigma (1 + \varepsilon p) \,,
  \end{align*}
where $\varepsilon$ is the acousto-electrical coupling constant. The difference in the electric boundary measurements between the perturbed and unperturbed situation is quantified by the power difference  
	\begin{equation*}\label{eq_pwd_obtaining}
		\spr{f_\varepsilon-f,g}=-\varepsilon\intV{p(x,t)\sigma\nabla u \cdot \nabla u_\varepsilon} \,,
	\end{equation*}
that can be computed from the measured boundary data $g,f,f_\varepsilon$. Here $u_\varepsilon$ is a solution of \eqref{eq_1} and \eqref{eq_neu} with $\sigma$ replaced by $\sigma_\varepsilon,$ and $f_\varepsilon = u_\varepsilon|_{\partial \Omega}$. Assuming that $\varepsilon$ is small allows the approximation $\sigma\nabla u \cdot \nabla u_\varepsilon \approx \sigma \abs{\grad u}^2$, and thus, by solving the equation
	\begin{equation*}
  		\spr{f_\varepsilon-f,g}=-\varepsilon\intV{p(x,t)\sigma|\nabla u|^2} \,,
	\end{equation*}
the interior power density \eqref{eq_powden} can be computed. Depending on the waves $p(x,t)$ the actual computation of $E(\sigma)$ might be an ill-posed problem. A similar derivation can be done for \eqref{eq_1} supplemented with \eqref{eq_dir}.

It is well known that a single measurement of the power density $\sigma \abs{\grad u(\sigma)}^2$ is in general not enough to uniquely determine the conductivity $\sigma$ \cite{Bal_2013,Isakov_2006}. However, it was shown in \cite{Capdeboscq_Fehrenbach_Gournay_Kavian_2009} for the two dimensional case that if measurements
	\begin{equation}\label{measurement_cond}
		\kl{\sigma \abs{\grad u_1(\sigma)}^2\,,
		\sigma \abs{\grad u_2(\sigma)}^2\,,
		\sigma \, \grad u_1(\sigma) \cdot \grad u_2(\sigma) } \,,
	\end{equation} 
with
	\begin{equation}\label{determinant_cond}
		\det\kl{\grad u_1(\sigma),\grad u_2(\sigma)} \geq c > 0 \,,
	\end{equation}
are available, where $u_1$, $u_2$ are two solutions of \eqref{eq_1}, then $\sigma$ can be uniquely determined from those measurements. (Note that the third quantity in \eqref{measurement_cond} can be obtained from a third power density measurement by the polarization identity.) Similar results were also obtained for $3$ dimensions in \cite{Bal_Bonnetier_Monard_Triki_2013} and for arbitrary dimensions in \cite{Monard_Bal_2012}, see also \cite{Alberti_Capdeboscq_2018}. Hence, the reconstruction of $\sigma$ profits from multiple power density measurements. See also \cite{Alberti_Bal_Cristo_2017,Bal_Hoffmann_Knudsen_2017,Capdeboscq_2015} for more information about the choice of boundary conditions.

Under the assumptions \eqref{measurement_cond}, \eqref{determinant_cond}, the inverse problem is well-posed and one can expect to reconstruct the conductivity stably with high contrast and resolution; see  \cite{Bal_Naetar_Scherzer_Scotland_2013,Capdeboscq_Fehrenbach_Gournay_Kavian_2009,Ammari_Bonnetier_Capdeboscq_Tanter_Fink_2008,Hoffmann_Knudsen_2014} for some numerical implementations of the problem.

To model the scenario when only a part of the boundary is accessible to the electrostatic measurements we introduce the proper subset $\Gamma_1 \subset \partial \Omega$ and assume that the induced current field has $\text{supp}(g) \subset \Gamma_1.$ This assumption tacitly  enforces  a no flux condition on the inaccessible boundary $\Gamma_0 = \partial \Omega \setminus \Gamma_1.$ The main purpose of this paper is to study the influence of the size of $\Gamma_1$ on the quality of the reconstructions. This is related to \cite{Ammari_Garnier_Jing_2012}, in which the authors derive an analytic formula for reconstructing the conductivity in a specific limited-angle setting and give a simple numerical example. However, the derived formula depends on the exact limited-angle setting and, as the authors themselves mention, does not work for general conductivity distributions.

For EIT the problem of limited angle data (in that context known as partial data) is fairly well understood \cite{Bukhgeim_Uhlmann_2002,Kenig_Sjostrand_Uhlmann_2007,Knudsen_2006,Imanuvilov_Uhlmann_Yamamoto_2010}; and the instability is known to be severe \cite{Caro_2016}. We expect that a similar instability appears here and we want to see how the ill-posedness of the problem is affected by accessibility of the measurement boundary.   

In this paper we take a computational approach to the problem by formulating the inverse problem as a nonlinear operator equation
	\begin{equation}\label{Fx=y}
		F(\sigma) = E \,.
	\end{equation}
We provide the Fr\'echet derivative and its adjoint of the operator $F$ and approximate the solution using Landweber iteration. Numerical examples are presented focusing especially on the limited angle problem. Furthermore, a numerical ill-posedness quantification is performed, quantifying the expected reconstruction quality in various areas of the domain $\Omega$ in this case by considering the singular value decomposition of the linearized problem.

The paper is organized as follows: in Section~\ref{sec_MathPre} we recall the basic notation and important results from PDE theory for the problem \eqref{eq_1}, \eqref{eq_neu}. In Section~\ref{sect_inv_prob} we discuss the inverse problem \eqref{Fx=y}, showing that the operator $F$ is Frechet differentiable. Furthermore, we derive the Frechet derivative and the adjoint thereof. The results are generalized to multiple measurements of the power density. The regularization approach, which we apply for approximating the solution of the inverse problem \eqref{Fx=y}, is briefly outlined in Section~\ref{sec_Regular}. The idea on ill-posedness quantification of the problem is given in Section~\ref{sect_illp_quant}. In Sections~\ref{sec_NumDetails} and \ref{sect_num_res} we describe the setting of our numerical example problem and present various reconstruction results for different boundary settings, especially focusing on the limited angle case. Moreover, we present results of the ill-posedness quantification.

% % % % % % % % % %
% Forward Problem %
% % % % % % % % % %
\section{Mathematical Preliminaries}\label{sec_MathPre}

In this section we recall the basic notations and results for the Neumann problem \eqref{eq_1}, \eqref{eq_neu}. In addition we consider the Fr\'echet differentiablity of the solution $u$ with respect to $\sigma.$  We start by stating the main assumptions taken throughout:
\begin{assumption}\label{ass_main}
Let $\Omega$ denote a non-empty, bounded, open and connected set in $\R^N$, $N=2,3$, with boundary $\bO \in C^{1,1}$. Furthermore, assume that $g \in L^2(\partial \Omega)$ is given such that
	\begin{equation}\label{cond_compatibility}
		\intS{g} = 0 \,.
	\end{equation}
Finally, we assume that a priori a lower bound $\sib > 0$ is given such that 
	\begin{equation}\label{Msb}
	  \sigma \in \Msb := \left\{ \sigma \in \LiO \vl \sigma \geq \sib  > 0   \right\} \,.
	\end{equation}
\end{assumption}

It is well-known from standard theory for elliptic PDEs \cite{Gilbarg_Trudinger_1998} that under Assumption \ref{ass_main} the Neumann problem \eqref{eq_1}, \eqref{eq_neu} has a unique weak solution 
	\begin{equation*}
	  u(\sigma) \in \HdoO : = \left\{ u \in H^1(\Omega) \, \Bigg\vert \, \intV{u}  = 0 \right\} \,.
	\end{equation*}
We occasionally drop $\sigma$ in the notation and write $u = u(\sigma).$ Moreover, there is a constant $C>0$ such that 
	 \begin{equation*}
	   \norm{u}_\HoO \leq C \norm{g}_{L^2(\partial\Omega)}\,.
	 \end{equation*}
If in addition  $\sigma \in C^{0,1}(\Omega)$ and $g\in H^{1/2}(\partial\Omega)$ then  $u \in \HtO$ with
	\begin{equation*}
	  \norm{u}_\HtO \leq C \norm{g}_{H^{1/2}(\partial\Omega)} \,.
	\end{equation*}

We now consider the solution mapping $u\colon \sigma \mapsto u(\sigma)$ as a mapping $\Msb \to \LtO$. From the weak formulation of the PDE problem the continuity estimate
	 \begin{equation*}
	   \norm{u(\sigma)- u(\sigma_0)}_\HoO \leq \cLM \norm{\sigma-\sigma_0}_\LiO \norm{u(\sigma_0)}_\HoO \,, \quad \forall \, \sigma,\sigma_0 \in \Msb \,, 
	 \end{equation*}
follows. In addition, $u$ is Fr\'echet differentiable with derivative $u'(\sigma)h$ at $\sigma \in \Msb$ in direction $h$, given as the unique weak solution to the Neumann problem
	 \begin{equation}\label{udiff}
	   \begin{split}
	     \div{\sigma \grad (u'(\sigma)h)} &= - \div{h \grad u(\sigma)} \,, \qquad \text{in } \Omega \,,
	     \\
	     (\sigma \grad (u'(\sigma)h)) \cdot \n \vert_\bO &= 0 \,.
	   \end{split}
	 \end{equation}

% % % % % % % % % %
% Inverse Problem %
% % % % % % % % % %
\section{Fr\'echet Differentiability of the Forward Operator}\label{sect_inv_prob} 

In this section we consider the forward operator $F \colon \sigma \mapsto E(\sigma).$  We first analyse the mapping properties in the situation of a single boundary condition and show that $F$ is Fr\'echet differentiable. Then we generalize the results to more boundary conditions.

% % % % % % % % % % % % % % % % %
% The Multiple Measurement Case %
% % % % % % % % % % % % % % % % %
\subsection{The Single Measurement Case}\label{sect_single}

For $\sigma \in \Msb$, the power density is naturally considered as an element in  $L^1(\Omega)$, i.e.,
    \begin{equation}\label{def_F}
	\begin{split}
      F: \Msb & \to \LoO \,,
      \\
      \sigma & \mapsto E(\sigma)\,,
	\end{split}
  \end{equation}
where $\Msb$ and $E$ are defined by \eqref{Msb} and \eqref{eq_powden} respectively, but since $\LoO$ is not reflexive, solving \eqref{Fx=y} in $\LoO$ is not straightforward. By increasing the regularity of $\sigma$ we pose the problem in a better suited Hilbert space. We introduce the set
	\begin{equation}\label{def_Ds}
	  \Ds(F) :=  \HsO \cap \Msb\,,
	\end{equation}
and note that for $ s > N/2 + 1$ by Sobolev embedding $\Ds(F) \subset C^{0,1}(\overline\Omega)\cap \Msb, $ and hence $u(\sigma) \in  H^2(\Omega)$ leaving $E(\sigma) \in L^2(\Omega)$ by the H\"older inequality. Thus we can consider
  \begin{equation}\label{def_Fs}
	\begin{split}
      F: \Ds(F) \to \LtO \,,
	\end{split}
  \end{equation}	
and the equation \eqref{Fx=y} can be considered in the standard framework of nonlinear ill-posed problems in Hilbert spaces \cite{Engl_Hanke_Neubauer_1996}.

We eventually address \eqref{Fx=y} using an iterative approach and hence the Fr\'echet derivative is required. In the following proposition we obtain  the derivative. The proof is analogous to the case of Dirichlet boundary conditions \cite{Bal_Naetar_Scherzer_Scotland_2013}.
\begin{proposition}\label{thm_DF}
The operator $F:\Ds(F) \to \LtO$ defined by \eqref{def_Fs} is Fr\'echet differentiable for $s>N/2+1$ with
	\begin{alignat}{2}\label{def_DF}
		F'(\sigma)h &= h \abs{\grad u(\sigma)}^2 && + 2 \, \sigma \grad u(\sigma) \cdot \grad(u'(\sigma)h)\,,
	\end{alignat}
where $u'(\sigma)h$ is defined by \eqref{udiff}.
\end{proposition}
\begin{proof}
This follows immediately from the definition of the operator, \eqref{udiff}  and the product and the chain rule applied to the function $x \abs{\grad f(x)}^2$, in the same way as in \cite{Bal_Naetar_Scherzer_Scotland_2013}.
\end{proof}

In order to calculate the adjoint of the Fr\'echet derivative of $F$, we need the following proposition regarding the adjoint of  embedding operators in Sobolev spaces.
\begin{proposition}
Denote by $\Es : \HsO \to \LtO$ the embedding operator for $s\ge0$, i.e., $\Es v = v$ for all $v \in \HsO$. Then for any element $w \in \LtO$ the adjoint $\Es^* w$ is given as the unique solution of the variational problem 
	\begin{equation}\label{def_Es}
		\spr{\Es^*w,v}_\HsO = \spr{w,v}_\LtO \,, \qquad \forall \, v \in \HsO \,. 
	\end{equation} 
\end{proposition}
\begin{proof}
This follows from the definition of $\Es$ and the Lax-Milgram Lemma.
\end{proof}
 
We are now prepared to give the adjoint of the Fr\'echet derivative of $F$:
\begin{theorem}\label{thm_adjoint_Fs}
Let $F:\Ds(F) \to \LtO$ be defined by \eqref{def_Fs} with $s>N/2 + 1$. Then for the adjoint of the Fr\'echet derivative of $F$ there holds
	\begin{equation}\label{eq_adjoint}
		F'(\sigma)^*w = \Es^*\left(w\abs{\grad u(\sigma)}^2 + 2\nabla u(\sigma)  \cdot \nabla (Aw)\right)\,,
	\end{equation}
where $A w \in V$ is given as the unique solution of the variational problem
  	\begin{equation} \label{def_A}
    	\intV{\sigma \nabla(A w) \cdot \nabla v} =  -\intV{  \sigma w \nabla u(\sigma) \cdot \nabla v}\,, \qquad \forall \, v \in V \,.
  	\end{equation}
\end{theorem}
\begin{proof}
By Proposition~\ref{thm_DF} we have 
  \begin{equation*}
   \begin{split}
     & \spr{F'(\sigma)h, w}_{\LtO} = \spr{h \abs{\grad u(\sigma)}^2 +2 \sigma \nabla u(\sigma) \cdot \nabla (u'(\sigma)h), w}_{\LtO}
     \\
     & \qquad= \spr{h,w \abs{\grad u(\sigma)}^2 }_{\LtO} + 2 \intV{ \sigma w \nabla u(\sigma) \cdot \nabla (u'(\sigma)h) }  \,.
   \end{split} 
  \end{equation*}
Together with \eqref{def_A} and \eqref{udiff}, there follows 
  \begin{equation*}
   \begin{split}
 	& \intV{ \sigma w \nabla u(\sigma) \cdot \nabla (u'(\sigma)h) }  
	= - \intV{\sigma \nabla (A w) \cdot \nabla (u'(\sigma)h)}  
     \\
     & \qquad = \intV{ h \nabla u(\sigma) \cdot \nabla (A w) }  \,,
   \end{split} 
  \end{equation*}
which, together with \eqref{def_Es} implies
  \begin{equation*}
   \begin{split}
 	 & \spr{F'(\sigma)h, w}_{\LtO} = \spr{h, w\abs{\grad u(\sigma)}^2 + 2 \nabla u(\sigma) \cdot \nabla (A w)}_{\LtO}  
     \\
     & 
     \qquad= \spr{h, \Es^*\left(w\abs{\grad u(\sigma)}^2  + 2 \nabla u(\sigma) \cdot \nabla (A w) \right)}_{\HsO}\,,
   \end{split} 
  \end{equation*}
which yields the assertion.
\end{proof}

\begin{remark}
  If $s$ is an integer, we can also consider the following inner product on $\HsO$
	\begin{equation*}
		\sspr{u,v} := \sum\limits_{\abs{\alpha}\leq s} \beta_\alpha \spr{\partial^\alpha u, \partial^\alpha v}_\LtO  \,,
	\end{equation*}
where $\{\beta_\alpha \}$ is a family of positive weights. The resulting inner product generalizes the standard inner product $\spr{.\,,.}_\HsO$ and induces an equivalent norm on $\HsO$. The adjoint of the operators $F : \Ds(F) \to \LtO$ with respect to these inner products can be computed in the same way as in Theorem~\ref{thm_adjoint_Fs}, with $\Es^*$ replaced by $\Esb^*$, where $\Esb^*w \in \HsO$ is given as the unique solution of the variational problem
	\begin{equation}\label{def_Esb}
		\sspr{\Esb^*w,v} = \spr{w,v}_\LtO \,, \qquad \forall \, v \in \HsO \,.
	\end{equation}
Using this weighted inner product gives us more flexibility in the reconstruction process, as we can put emphasis on different derivatives of the solution. A similar generalization of the scalar product is also possible for $\HsO$ with $s\in\R$.
\end{remark}

% % % % % % % % % % % % % % % % %
% The Multiple Measurement Case %
% % % % % % % % % % % % % % % % %
\subsection{The Multiple Measurement Case}\label{sect_multi}

As mentioned in the introduction, having the internal power density for one boundary condition is in general not sufficient to uniquely reconstruct the conductivity. To consider multiple data we introduce  $\{g_j\}_{j=1}^M$ of boundary current data such that  $g_j \in H^{\frac{1}{2}}(\bO)$, for $j \in \{1,\dots,M\}$ where $M \in \N$ is fixed. Furthermore, denote by $E_j$ the power density
		\begin{equation*}
			E_j(\sigma) := \sigma \abs{\grad u_j(\sigma)}^2 \,,
		\end{equation*}
	where $u_j(\sigma)$ is the weak solution of the boundary value problem
		\begin{equation}\label{prob_forward_multi}
		  \begin{split}
		      -\div{\sigma \nabla u_j} & = 0 \,, \quad \text{in} \; \Omega \,, \\
		      (\sigma \nabla u_j) \cdot \n \,|_{\bO} & = g_j \,.
		  \end{split}	
		\end{equation}

This problem can again be written as a nonlinear inverse problem in standard form, or rather, as a nonlinear system in standard form, by introducing the nonlinear operator
	\begin{equation}\label{def_F_multi}
	\begin{split}
		F \,:\, \Ds(F)  \to \LtO^M \,, 
		\quad
		\sigma \mapsto \Kl{E_j(\sigma)}_{j=1}^M \,.
	\end{split}
	\end{equation}

Continuity and Fr\'echet differentiability readily translate from $F$ \eqref{def_F} in the single measurement case to $F$ defined by \eqref{def_F_multi}. For example, for the Fr\'echet derivative we have
	\begin{equation}\label{def_DF_multi}
	\begin{split}
		F'(\sigma)h := \Kl{ h \abs{\grad u_j(\sigma)}^2 + 2\, \sigma \grad u_j(\sigma) \cdot \grad( u_j'(\sigma)h) }_{j=1}^M \,,
	\end{split}
	\end{equation}
with $u_j'(\sigma)h$ being given analogously as in \eqref{udiff}, and for the adjoint we have
	\begin{equation}\label{def_DF_adj_multi}
	\begin{split}
		F'(\sigma)^*w := \sum\limits_{j=1}^{M} \Es^*\kl{ w_j \abs{\grad u_j(\sigma)}^2 + 2\, \sigma \grad u_j(\sigma) \cdot \grad( Aw_j)}\,.
	\end{split}
	\end{equation}

% % % % % % % % % % % % % %
% Regularization Approach %
% % % % % % % % % % % % % %
\section{Iterative Regularization Approach} \label{sec_Regular}

Both the single and the multiple measurement problems of the previous section are inverse problems in the standard form
	\begin{equation*}
	F(x) = y\,,
	\end{equation*}
and therefore, need to be regularized in order to enable a stable reconstruction of the conductivity $\sigma$ from noisy measurement data $\Ed$. Besides the well-known Tikhonov regularization and its variants \cite{Engl_Hanke_Neubauer_1996}, iterative regularization methods are also very popular, especially for nonlinear Inverse Problems \cite{Kaltenbacher_Neubauer_Scherzer_2008}. Since the focus of this paper lies more on qualitative and quantitative aspects of the solution and less on numerical efficiency, we focus on the following simple yet robust Landweber-type gradient method, given by
    \begin{equation}\label{Landweber}
        \begin{split}
            \xkpd = \xkd + \okd\kl{\xkd} \skd\kl{\xkd}
            \,,
            \\
            \skd\kl{x} := F'\kl{x}^*\kl{\yd - F\kl{x}} \,,
        \end{split}
    \end{equation}
where for the stepsize $\okd$ we use the steepest descent stepsize \cite{Scherzer_1996} 
    \begin{equation}\label{stepsize}
        \okd(x) := \frac{\norm{\skd\kl{x}}^2 }{\norm{F'(x)\skd(x)}^2} \,.
    \end{equation}
As a stopping criterion, we employ the well-known Morozov discrepancy principle \cite{Morozov_1984}, i.e., the iteration is stopped after $\ks$ steps, with $\ks$ satisfying
    \begin{equation}\label{discrepancy}
        \norm{\yd - F\kl{\xksd}} \le \tau \delta \le \norm{\yd - F\kl{\xkd}}\,, \qquad 0\le k \le k_*\,,
    \end{equation}
where $\tau$ is an appropriately chosen positive number ($\tau \in [1,2]$ being common practise) and $\delta$ is the error level satisfying the error estimate $\norm{y-\yd} \le \delta$.

\begin{remark}
Note that for proving the convergence of iterative regularization methods one requires at least a weak form of the so-called nonlinearity or tangential cone condition (see \cite{Kaltenbacher_Neubauer_Scherzer_2008} for details). This condition is to the best of our knowledge  not known for this particular problem.
\end{remark}

% % % % % % % % % % % % % % % % %
% Ill-Posedness Quantification  %
% % % % % % % % % % % % % % % % %
\section{Ill-Posedness Quantification}\label{sect_illp_quant}

In order to get a better understanding of the reconstruction quality in different areas of the domain, we also consider an ill-posedness quantification of the problem based on the singular value decomposition (SVD) of the discretization of the Fr\'echet derivative of $F$ at the exact solution $\sD$.

For linear operators $F$, the degree of ill-posedness of the inverse problem $F(x) = y$ is directly connected to the singular value expansion of $F$ \cite{Engl_Hanke_Neubauer_1996}, a rapid decay of the singular values corresponding for example to severe ill-posedness of the problem.
In the nonlinear case, the connection between the ill-posedness and the Fr\'echet derivative $F'(x)$ is not as strong as one might expect it to be (see for example \cite{Schock_2002,Engl_Kunisch_Neubauer_1989}). However, in many cases there is a connection, as can for example be seen from the assumption
	\begin{equation}
		\norm{F'(\xD)h}_Y \geq c \norm{h}_{-a} \,, \qquad \forall \, h \in X \,,
	\end{equation} 
commonly used for analyzing iterative methods in Hilbert scales \cite{Neubauer_2000}. Here the parameter $a$ effectively measures the degree of ill-posedness of the problem. Furthermore, since almost all methods for solving ill-posed problems rely on the Fr\'echet derivative of $F$, information about the expectable quality of the reconstruction may be obtained from this Fr\'echet derivative.

Given the two finite element basis $\Kl{\phi_i}$ and $\Kl{\psi_i}$ of the data and the image space of $F$ used in the discretization of the inverse problem, the transfer matrix $T$ of the discretization of the Fr\'echet derivative of $F$ is given by
	\begin{equation}\label{transfer_matrix}
		T_{i,j} := \spr{F'(\sD)\phi_i,\psi_j}_\LtO \,.
	\end{equation}
In Section~\ref{sect_num_res}, we compute $T$ and its SVD for different boundary condition settings corresponding to various parts of the boundary being inaccessible for measurements. The resulting singular values and singular vectors are then analyzed and correlated to the obtained reconstructions for each considered setting.

% % % % % % % % % % % % % % % % %
% Numerical Algorithm           %
% % % % % % % % % % % % % % % % %
\section{Numerical Algorithm}
In Section \ref{sec_Regular}, we outlined a regularization approach for solving the inverse problem \eqref{Fx=y} which is based on the Landweber-type iterartion \eqref{Landweber}. In this section, we shortly describe how this algorithm is implemented for a single measurement.
In pseudocode notation it takes the following form:

	\begin{algorithm}[H]
		\KwData{Power density data $E^\delta$.}
		\KwIn{Initial guess $\sigma_0$, parameter $\tau$, noise level $\delta$.}
		\KwResult{Reconstructed conductivity $\sigma_{k^*(\delta, \Ed)}$.}
    	\Begin{
		$k \gets 0$
		\newline
		$\sigma_k \gets \sigma_0$
		\newline
		\Repeat{Residual norm $\norm{E^\delta - F(\sigma_k)} \le \tau \delta$}
		{
			Find the potential $u_k$ as a solution to \eqref{eq_1}, \eqref{eq_neu} with $\sigma = \sigma_k$.
			\newline
			Calculate the power density $F(\sigma)$ using \eqref{eq_powden}  with $u=u_k$, $\sigma = \sigma_k$.
			\newline
			Find $Aw$ as a solution of \eqref{def_A} with $u=u_k$, $w=\Ed-F(\sigma)$, $\sigma = \sigma_k$.
			\newline
			Calculate $F'(\sigma)^*w$ using \eqref{eq_adjoint} with $u = u_k$, $Aw$, $w=\Ed-F(\sigma)$, $\sigma = \sigma_k$, and solving \eqref{def_Es}. 
			\newline
			Calculate stepsize $\okd$ in several steps:
			\newline
			Find $u'(\sigma)h$ as a solution of \eqref{udiff} with $h = F'(\sigma)^*w$, $\sigma = \sigma_k$.
			\newline
			Calculate $F'(\sigma)h$ using \eqref{def_DF} with $u = u_k$, $u'(\sigma)h$, $h = F'(\sigma)^*w$, $\sigma = \sigma_k$.
			\newline
			Calculate $\okd$ using \eqref{stepsize}.			
			\newline
			Update $\sigma_{k+1}=\sigma_k+\okd F'(\sigma_k)^*w$.
			\newline
			$k \gets k+1$
		}
		}
		
	\caption{Reconstruction of the electrical conductivity from a single measurement of the power density.}
	\end{algorithm}
\vspace{0.5cm}	
The variational problems can be solved by standard finite element approaches, see below for details. Obviously, the above algorithm can be generalized to the multiple measurement case.

% % % % % % % % % % % % % % % % % % % % % % % % %
% Numerical Setting and Implementation Details  %
% % % % % % % % % % % % % % % % % % % % % % % % %

\section{Numerical Setting and Implementation Details}\label{sec_NumDetails}

We now describe the precise setting of our numerical example problem. For the domain $\Omega$, we choose a unit disk in $2D$, i.e., in polar coordinates,
	\begin{equation*}
	\begin{split}
		\Omega := \Kl{(r,\theta) \in [0,1) \times [0,2\pi]} \,.
	\end{split}
	\end{equation*}
For the accessible boundary $\Gamma_1$ we choose the family of subsets $\Gamma(\alpha) \subset \bO$ defined by  
	\begin{equation*}
		\Gamma(\alpha) := \Kl{(r,\theta) \in \{1\}\times [0,\alpha]}\,,
	\end{equation*}
and we set
	\begin{equation}\label{bdc_theta}
    	g_j(r,\theta) := \sin\kl{\frac{2 j \pi \theta}{\alpha}} \,, 
    	\qquad \forall \, (r,\theta) \in \Gamma(\alpha) \,.
	\end{equation} 	
On the remaining part of the boundary, we always assume that $g_j = 0$.
The resulting boundary functions $g_j$ are continuous on $\Gamma$ and satisfy \eqref{cond_compatibility}.  The trigonometric functions \eqref{bdc_theta} are a natural choice for current density patterns \cite{Mueller_Siltanen_2012}.  Being normed, they represent elements of an orthonormal basis of the space $L^2(\Gamma(\alpha))$. Moreover, this choice of boundary functions guarantees a similar magnitude of the computed power densities $E_i$, which ensures that every power density contributes evenly to the reconstruction.

Note that for the choice of the accessible boundary $\Gamma_1$ we consider single closed intervals.  It would also, for example, be possible to choose $\Gamma_1$ as consisting of multiple disjoint intervals, but in any case, the effect of various limited angle cases can already be observed in the single interval setting considered here.

For the true conductivity $\sD$ we use the phantom depicted in Figure~\ref{fig_phantom}. It has a uniform background of value $1$ as well as three inclusions: two circular inclusions of magnitude $1.3$ and $2$, respectively, and a crescent shaped inclusion of magnitude $1.7$, which are slightly smoothed towards their edges to conform with the smoothness requirements, since due to \eqref{def_Ds}, for $\sigma$ to be in $\Ds(F)$ it has to be $\Ht$ smooth. In order to implement this, we use $2$D bump functions built from piecewise polynomial functions, where the polynomials are chosen in such a way that the resulting bump function is $C^2$.
	\begin{figure}[H] \centering
		\includegraphics[width=0.36\textwidth]{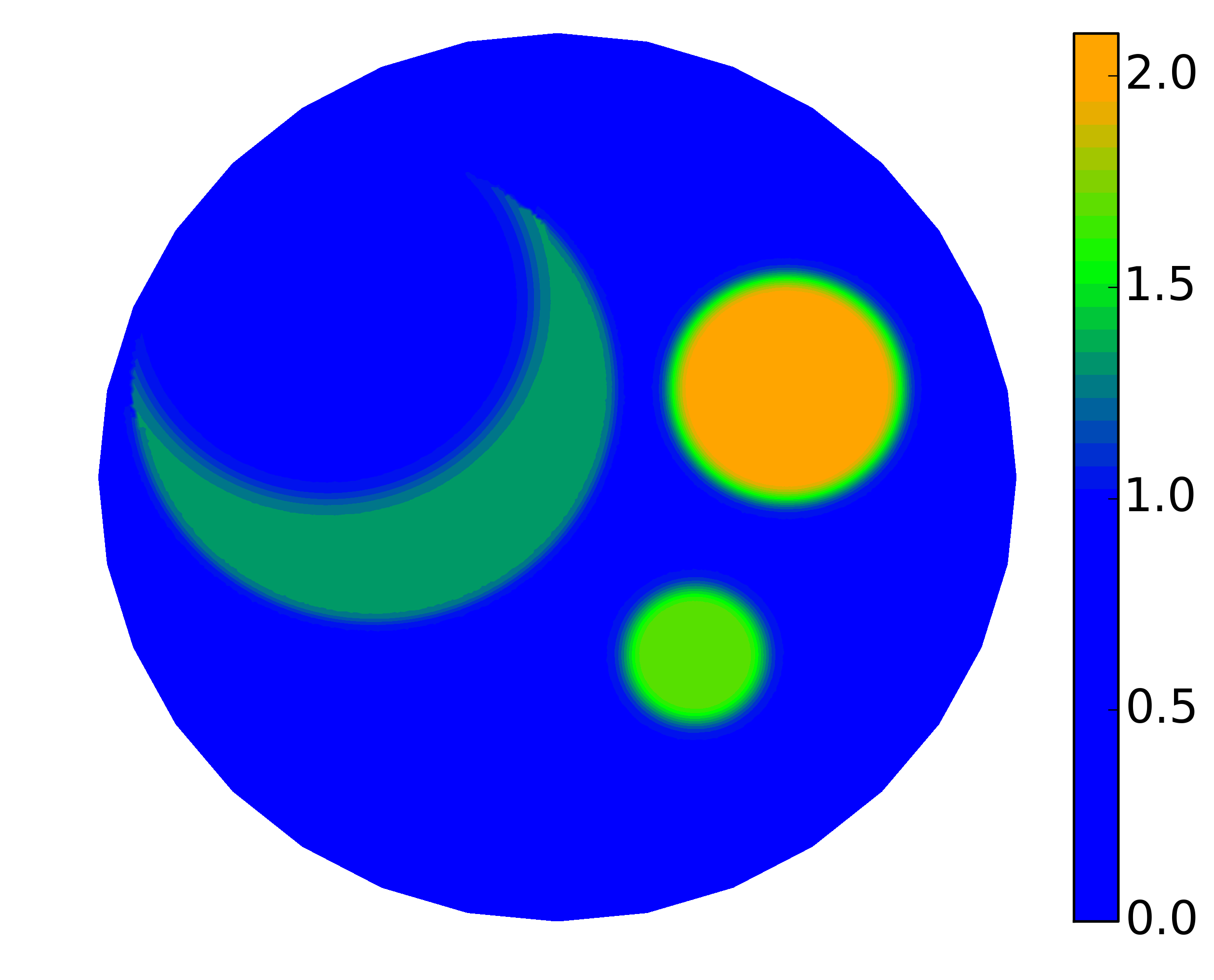}
		\caption{Exact value of the electrical conductivity $\sD$.}
		\label{fig_phantom}
	\end{figure}  

The discretization, implementation and computation of the involved variational problems was done using Python and the library FEniCS \cite{Alnaes_Blechta_2015}. A triangulation with approximately $2000$ vertices for discretizing the domain was used. This rather coarse choice of the discretization is due to time limitations in the computation of the SVD, since computing the matrix \eqref{transfer_matrix} already takes approximately $5$ hours for this discretization level, see Section~\ref{sect_num_res}. The power density data $E(\sD)$ was created by applying the forward model to $\sD$ using a finer discretization with approximately $40000$ vertices to avoid an inverse crime. The resulting power densities are depicted in Figure~\ref{fig_powden_smooth_100_75_50} for the angles $\alpha=2\pi$, $\alpha= 3\pi/2$, and $\alpha = \pi$, respectively. The red circle (segment) in the figures indicate the available, i.e., non-zero, boundary. Accessibility of the boundary is reflected in the power densities: in Figure~\ref{fig_powden_smooth_100_75_50}, the angle $\alpha=2\pi$, we clearly see the internal structure such as the location of the inclusions, while for the angles $\alpha= 3\pi/2$ and $\alpha = \pi$ only some of it, but less than before, is visible. Furthermore, the potentials induced by the boundary functions $g_j$ for $j=2,3$ have a higher frequency and do not penetrate deep into the domain.
Different random noise with a relative noise level of $5 \%$ is added to the power density to obtain the noisy data $\Ed$, i.e., $\Ed = E + \delta^{rel} \norm{E} \tilde{e}/\norm{\tilde{e}}$, where $\tilde{e}$ is a normally distributed random noise vector and $\delta^{rel}$ is the relative noise level. Obviously, with this choice one has an absolute noise in the data of $\delta=\delta^{rel}\norm{E}$.

	\begin{figure}[H] \centering
		\includegraphics[width=0.9\textwidth, trim={3cm 19cm 0cm 3cm}, clip]{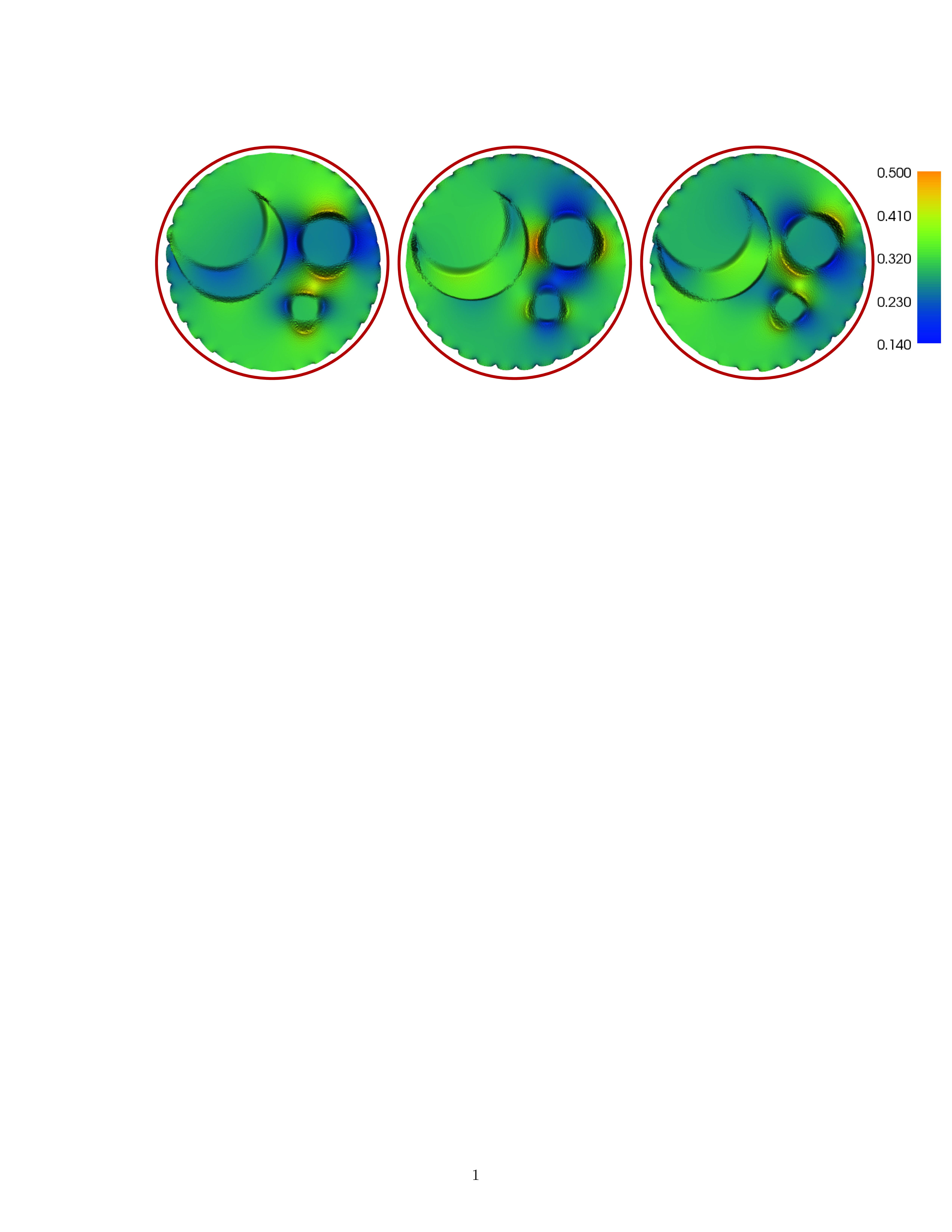}
		
		\includegraphics[width=0.9\textwidth, trim={3cm 19cm 0cm 3cm}, clip]{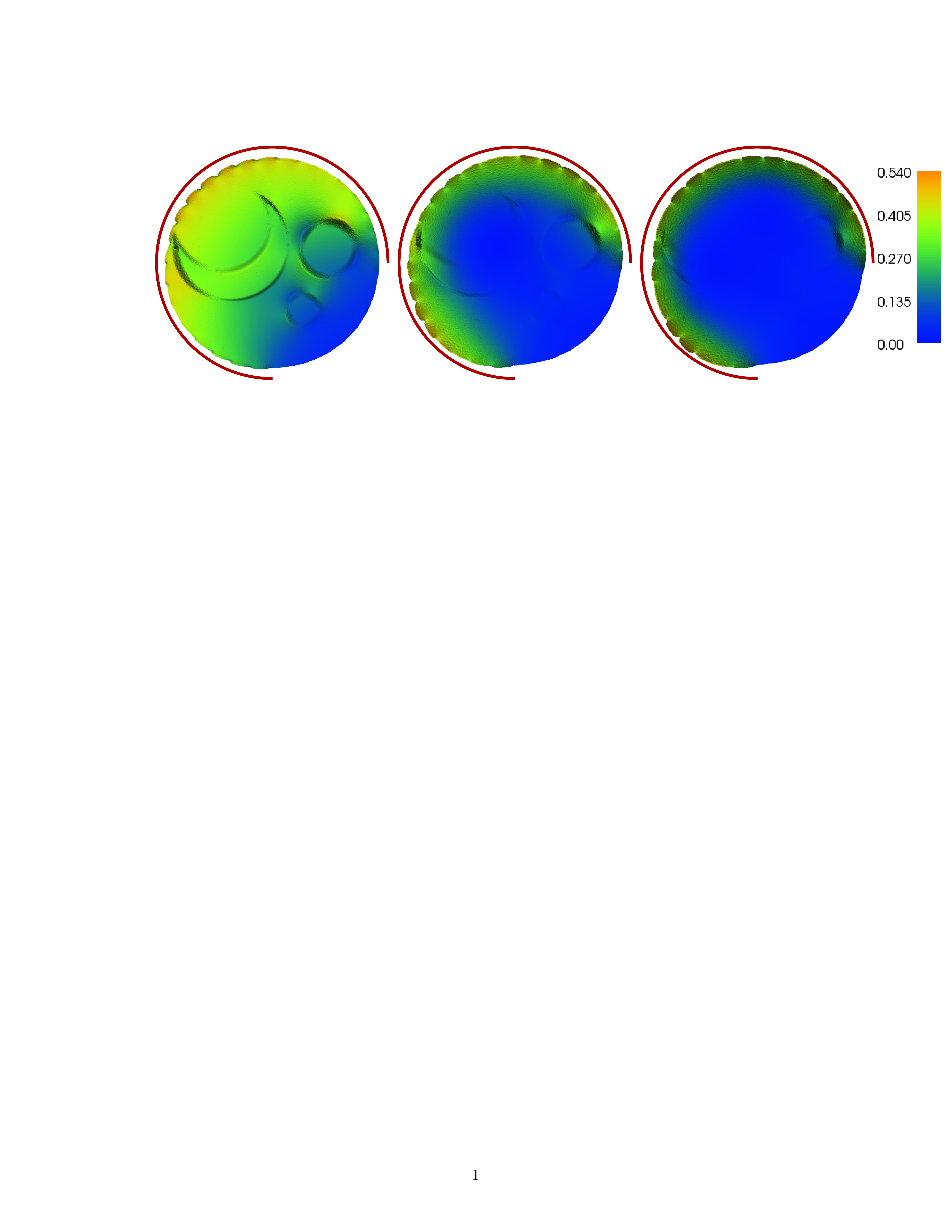}
		
		\includegraphics[width=0.9\textwidth, trim={3cm 19cm 0cm 3cm}, clip]{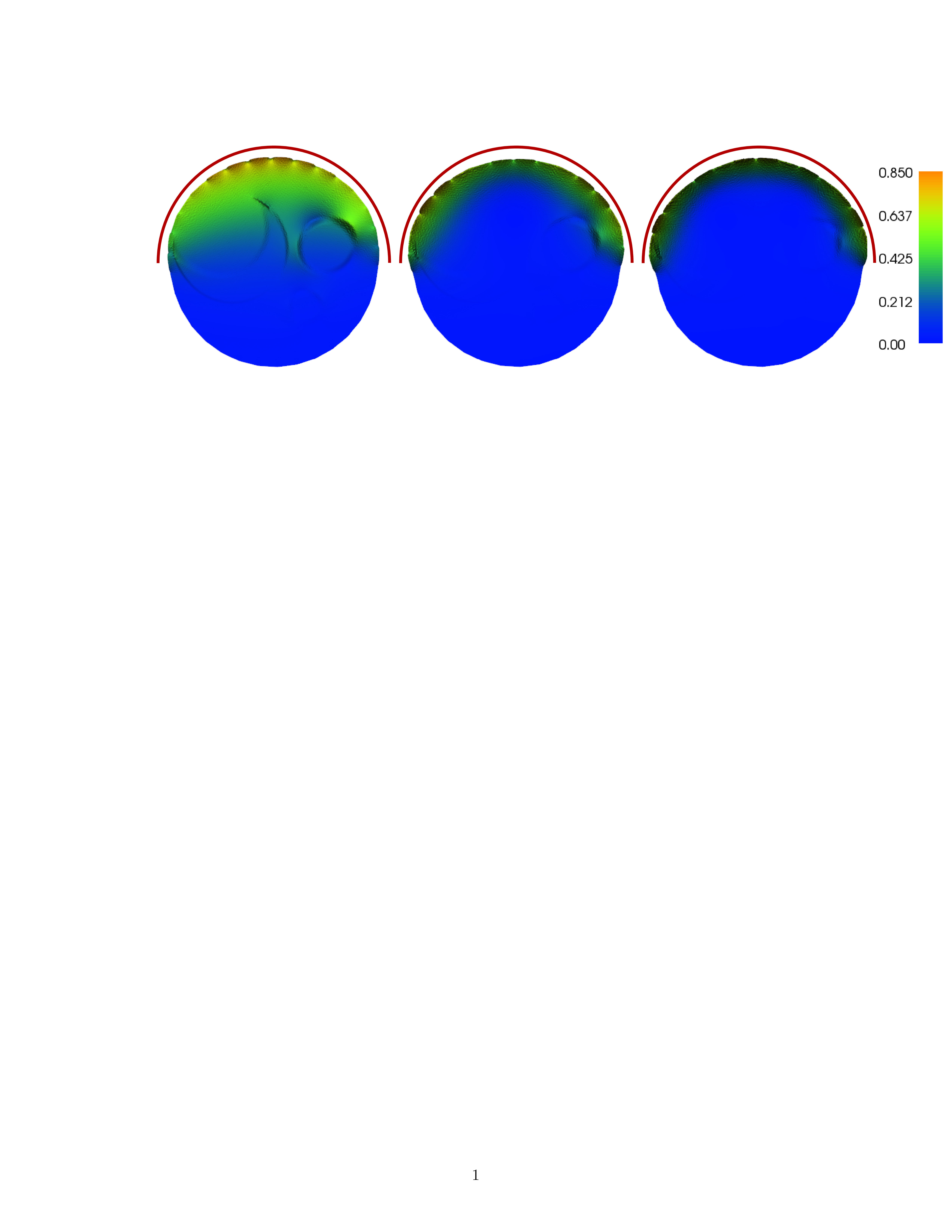}
		
		\caption{
		Power densities $E_j(\sD)$ with $\sD$ as in Figure~\ref{fig_phantom} with boundary data $g_j$, $j=1,2,3$ defined in: (first row) \eqref{bdc_special}; (second row) \eqref{bdc_theta} with $\alpha=3\pi/2$; (third row) \eqref{bdc_theta} with $\alpha=\pi$. The red curves indicate the support of $g_j$.}
		\label{fig_powden_smooth_100_75_50}
	\end{figure}
	
Since the domain $\Omega$ is two-dimensional, i.e., $N=2$, by the above analysis we should choose $s>2$ in the domain of $F$. However, since numerically there is hardly any difference between using $s=2$ and $s=2+\eps$ for $\eps$ small enough, and since $s$ should be kept as small as possible to avoid unnecessary smoothness requirements for the exact conductivity $\sD$, we choose $s=2$ for ease of implementation in the examples presented below. For obtaining the reconstructions, the steepest-descent Landweber method \eqref{Landweber} together with the discrepancy principle \eqref{discrepancy} with the canonical choice $\tau = 1$ was used. For the initial guess, $\sigma_0 = 1.5$ was used in all tests.

Furthermore, in all cases additional reconstructions are presented where instead of using $\Es$ in the adjoint of the Fr\'echet derivative the operator $\Esb$ defined by \eqref{def_Esb} was used with $s=2$ and the choice $\beta_\alpha = 1,10^{-3},10^{-6}$ for $\abs{\alpha}=0,1,2$, respectively. Moreover, we also present results in case that $\Es^*$ is dropped altogether in the reconstruction process, which can be seen as a preconditioning or in the light of regularization in Hilbert scales \cite{Neubauer_2000}. We refer to those cases as using the $\Htb$ or the $\Lt$ adjoint, while in the standard case we speak of using the $\Ht$ adjoint.

% % % % % % % % % % % %
% Numerical Examples  %
% % % % % % % % % % % %
\section{Numerical Results}\label{sect_num_res}
In this section we present various numerical results for different boundary value settings. Hereby, an emphasis is placed on the limited angle case, i.e., that $g = 0$ on the inaccessible boundary part $\bO\setminus\Gamma(\alpha)$. For ease of writing, we refer to these cases by the percentage value of the available boundary, e.g., we say that $75\%$ of the boundary is available for measurements if $\alpha = 3 \pi /2$. We consider the cases of $25\%$, $50\%$, $75\%$, and $100\%$  available boundary in this section. Moreover, we present an ill-posedness quantification of the problem based on the singular value decomposition of the Fr\'echet derivative of $F$ in Section~\ref{sect_illp_quant}.

% % % % % % % % % % % % % % % % %
% Reconstructions without noise %
% % % % % % % % % % % % % % % % %
\subsection{Reconstructions without Noise}

Before considering the noisy data case of interest to us, we first present two examples where no noise was added to the data. Since the discrepancy principle is not a suitable stopping rule in case of no noise, the iteration has to be stopped differently. Due to computational limitations and since the iterative procedure does not make much progress from this point onwards, the process was stopped after $1000$ iterations in both cases.

\begin{example}\label{ex_100_no_noise}
As a first test we look at the reconstruction of the conductivity for a fully available Neumann boundary and three power density measurements. Contrary to all the other tests, here we have a different set of boundary functions, namely  
	\begin{equation*}
  		g_{1}=\sin(\theta)\,, \qquad g_{2}=\cos(\theta)\,, \qquad g_{3}=(\sin(\theta)+\cos(\theta))/\sqrt{2} \,.
  		\label{bdc_special}
	\end{equation*}
After $1000$ iterations we obtain the reconstructions for the $\Lt$, $\Htb$ and $\Ht$ adjoint case depicted in Figure~\ref{fig_recon_100_no_noise}. The resulting reconstructions look rather similar, which is due to the fact that without noise, the residual $F(x)-y$ is already smooth and hence, the various smoothing properties of the different adjoints do not have much additional effect. However, they differ in the noisy case, where the $\Htb$ adjoint performs somewhat better than the others (see Section~\ref{sect_num_noise}).

	\begin{figure}[H] \centering
		\includegraphics[width=0.9\textwidth, trim={3cm 19cm 0cm 3cm}, clip]{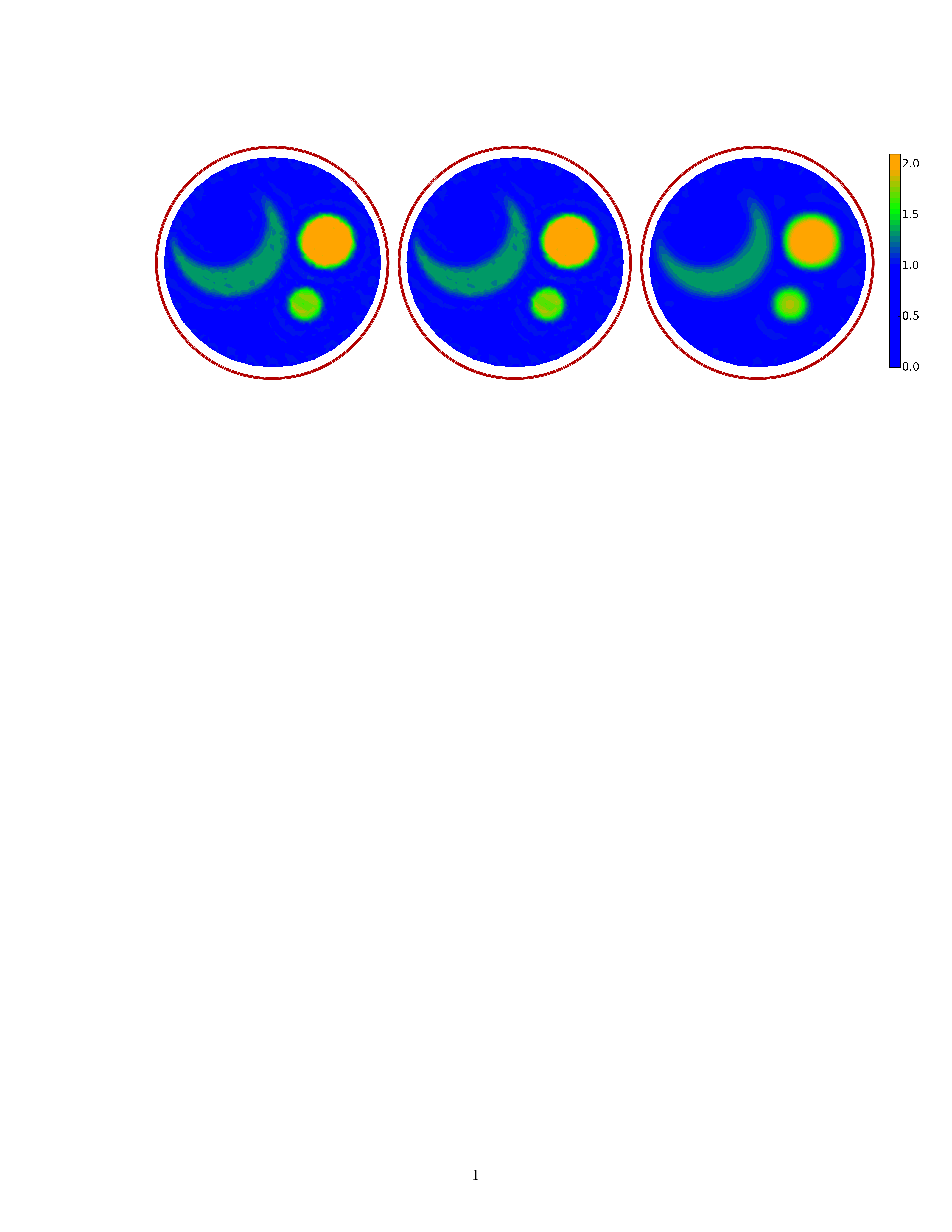}
		\caption{
		Reconstruction of conductivity $\sD$, Figure~\ref{fig_phantom}, with boundary data $g_j$, $j=1,2,3$ defined in \eqref{bdc_special}. The red curves indicate the support of $g_j$. From left to right: the $\Lt$, $\Htb$, and $\Ht$ adjoint is used.}
		\label{fig_recon_100_no_noise}
	\end{figure}
\end{example}

\begin{example}\label{ex_75_50_25_no_noise}
Following example \ref{ex_100_no_noise} we present reconstructions for $75\%, 50\%, 25\%$ boundary available for measurements with boundary data $g_j$, $j=1,2,3$ defined in \eqref{bdc_theta} and $\Htb$ adjoint, which are depicted in Figure~\ref{fig_recon_75_50_25_no_noise}.

	\begin{figure}[H] \centering
		\includegraphics[width=0.9\textwidth, trim={3cm 19cm 0cm 3cm}, clip]{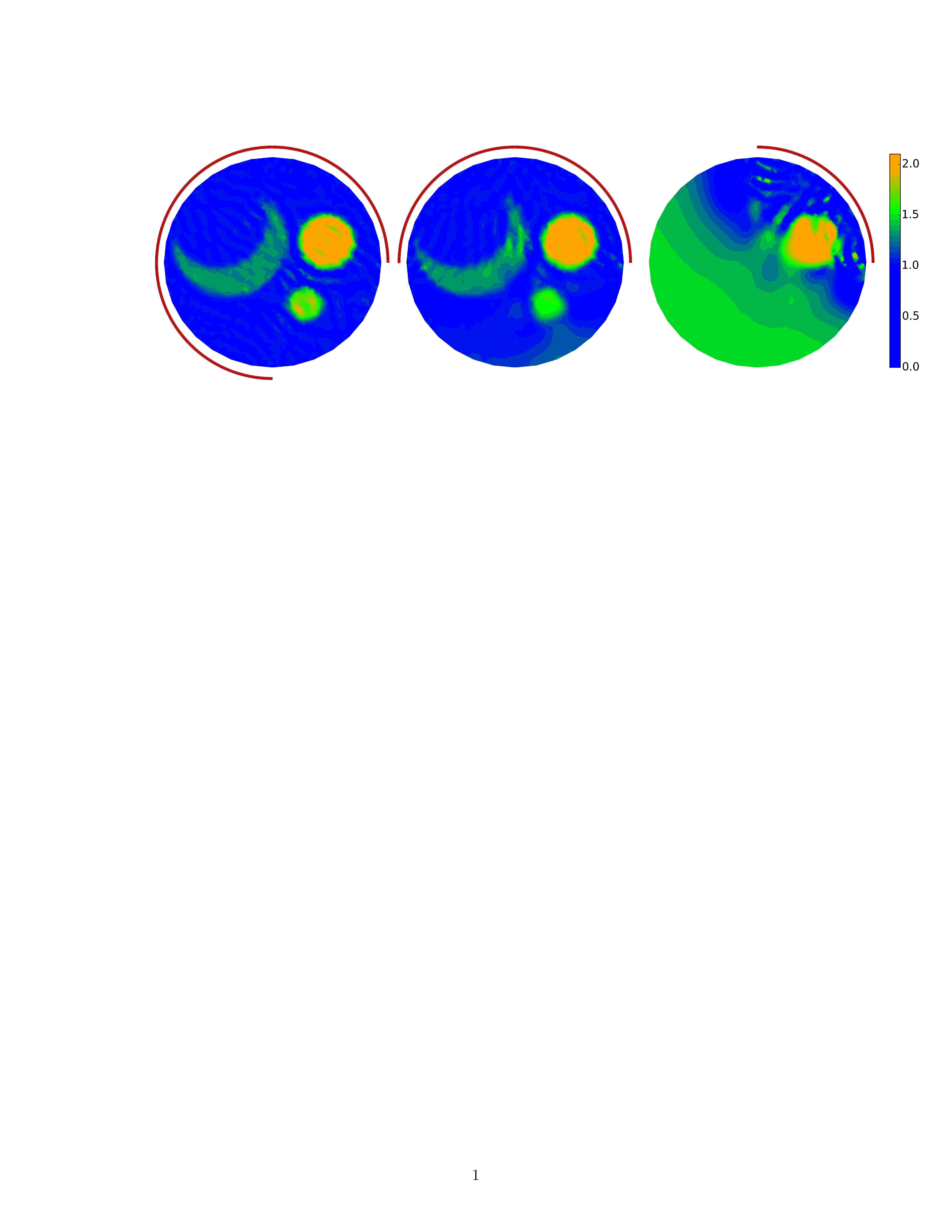}
		\caption{
		Reconstruction of conductivity $\sD$, Figure~\ref{fig_phantom}, with boundary data $g_j$, $j=1,2,3$ defined in \eqref{bdc_theta} from various limited angles. The red curves indicate the support of $g_j$. The $\Htb$ adjoint is used.}
		\label{fig_recon_75_50_25_no_noise}
	\end{figure}
\end{example}

% % % % % % % % % % % % % % % %
% Reconstructions with noise  %
% % % % % % % % % % % % % % % %
\subsection{Reconstructions with Noise}\label{sect_num_noise}

After we saw in the previous section that reasonable reconstructions can be obtained in the case of noise-free data, in this section we focus on noisy data $\Ed$ with a noise level of $\delta = 5\%$. Again the focus is on different limited angle cases.

\begin{example}\label{ex_100}
We consider $100 \%$ boundary available for measurements with boundary data $g_j$, $j=1,2,3$ defined in \eqref{bdc_special}. The iteration terminated after $3$, $3$ and $74$ iterations for the $\Lt$, $\Htb$ and $\Ht$ adjoint case, respectively, and yielded the reconstructions depicted in Figure~\ref{fig_recon_100_75_50_25}. Even though the noise level is high, the conductivity $\sD$ is nicely reconstructed both in shape and quantity. The $\Lt$ adjoint does not give enough smoothness on the solution, which is visible in the non-sharp edges of the inclusions. Due to the high noise level, the discrepancy principle stops the iteration very early, which affects the contrast of the reconstructions. 
\end{example}

\begin{example}\label{ex_75}
Next we consider $75 \%$ boundary available for measurements with boundary data $g_j$, $j=1,2,3$ defined in \eqref{bdc_theta}. In this case the iteration stops after $16$, $10$ and $177$ steps for the $\Lt$, $\Htb$ and $\Ht$ adjoints, respectively, which leads to the reconstructions depicted in Figure~\ref{fig_recon_100_75_50_25}. As we can see, the missing data in the right bottom part of the power density in the Figures~\ref{fig_powden_smooth_100_75_50} (second row) transfers to the reconstructed conductivity through artefacts near the $\bO\setminus\Gamma(\alpha)$ boundary, where the background value and inclusions are not well reconstructed. Similarly to the previous example, the solution lacks smoothness with the $\Lt$ adjoint, but captures more of the internal structure compared to the $\Ht$ adjoint, which hardly detects the small circular inclusion. Meanwhile, the $\Htb$ adjoint exhibits a good trade-off result between the other two. 
\end{example}

\begin{example}\label{ex_50}
For $50 \%$ available boundary and three measurements we obtain the reconstructions depicted in Figure~\ref{fig_recon_100_75_50_25}. The discrepancy principle was satisfied after $44$, $38$ and $602$ iterations for the $\Lt$, $\Htb$ and $\Ht$ adjoints, respectively. In this test we see what happens when only half of the boundary is accessible and hence, half of the internal conductivity can be reconstructed, see Figure~\ref{fig_powden_smooth_100_75_50} (third row). The reconstructions are worse than in the previous examples, although we are able to obtain some information about the inclusions. The conductivity value of the big circle comes closer to the expected value and its shape remains  almost proper, while the crescent is only partly visible. The small circular inclusion cannot be reconstructed due to the lack of information in this area.
\end{example}

\begin{example}\label{ex_25}
As a last test we consider an available boundary of only $25 \%$ with three measurements. We obtain the reconstructions depicted in Figure~\ref{fig_recon_100_75_50_25} after $1000$ iterations (the iteration was terminated even though the discrepancy principle was not reached due to time limitations). We can recover the big circle inclusion located close to the accessible boundary with some artefacts visible around it for the cases of the $\Lt$ and $\Htb$ adjoints. The $\Ht$ adjoint has a strong smoothing effect, which reduces the artefacts in the solution. Interestingly, even though only the large circle inclusion is recovered, this has a higher contrast than in the previous examples with noise.
\end{example}

	\begin{figure}[H]
		\centering
		\parbox{\textwidth}{\hspace{2cm} $\Lt$ adjoint \hspace{2cm} $\Htb$ adjoint \hspace{2cm} $\Ht$ adjoint}
		\includegraphics[width=0.9\textwidth, trim={3cm 19cm 0cm 3cm}, clip]{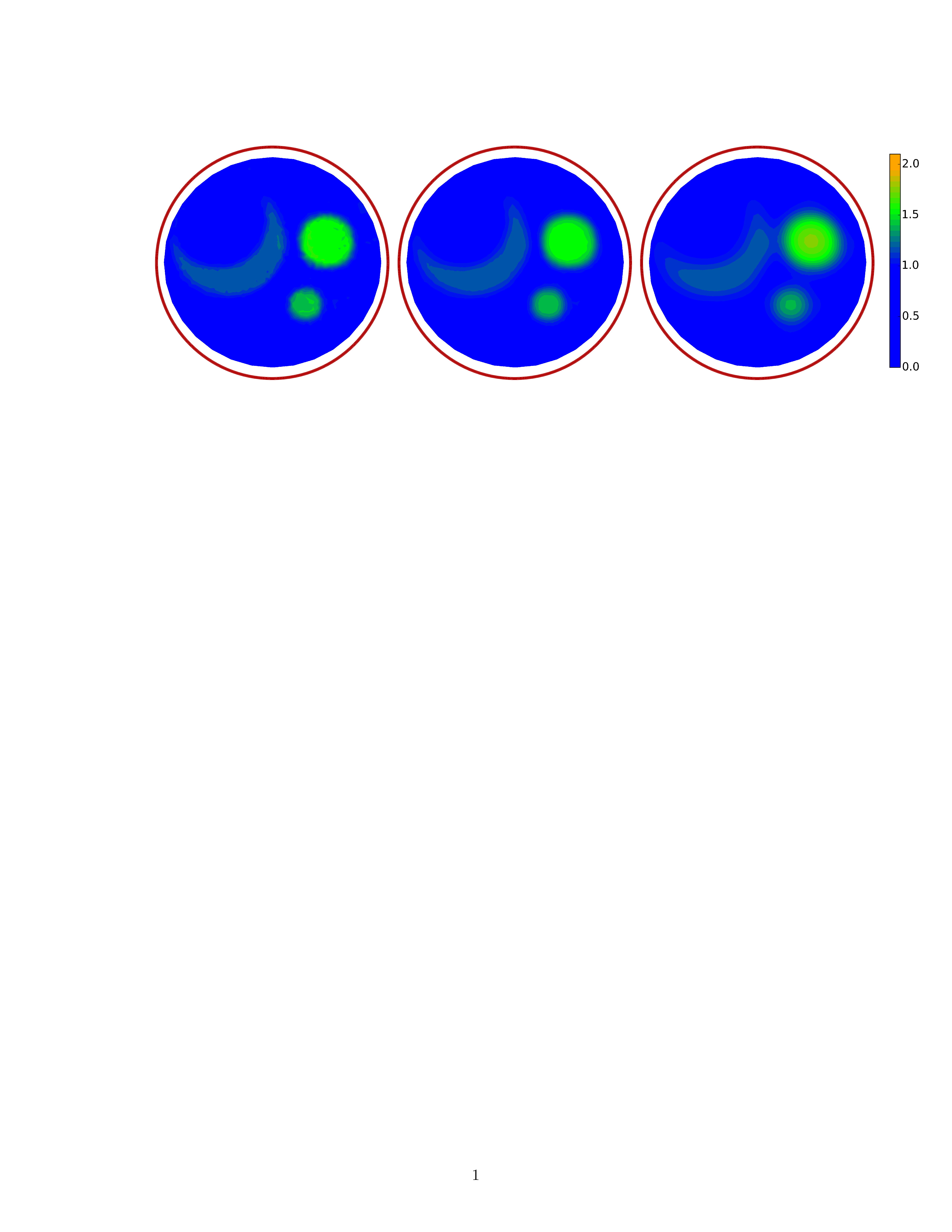}

		\includegraphics[width=0.9\textwidth, trim={3cm 19cm 0cm 3cm}, clip]{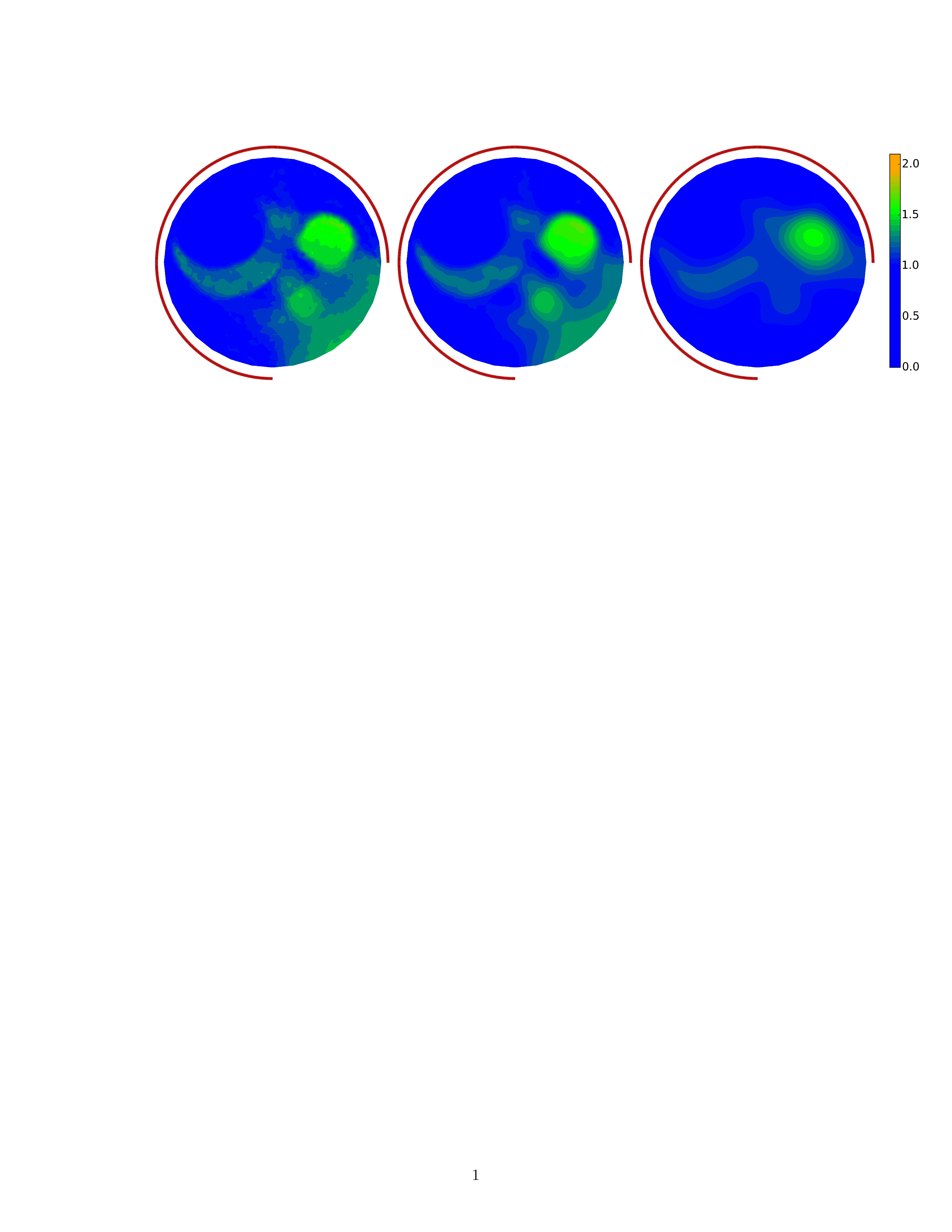}
		
		\includegraphics[width=0.9\textwidth, trim={3cm 19cm 0cm 3cm}, clip=true]{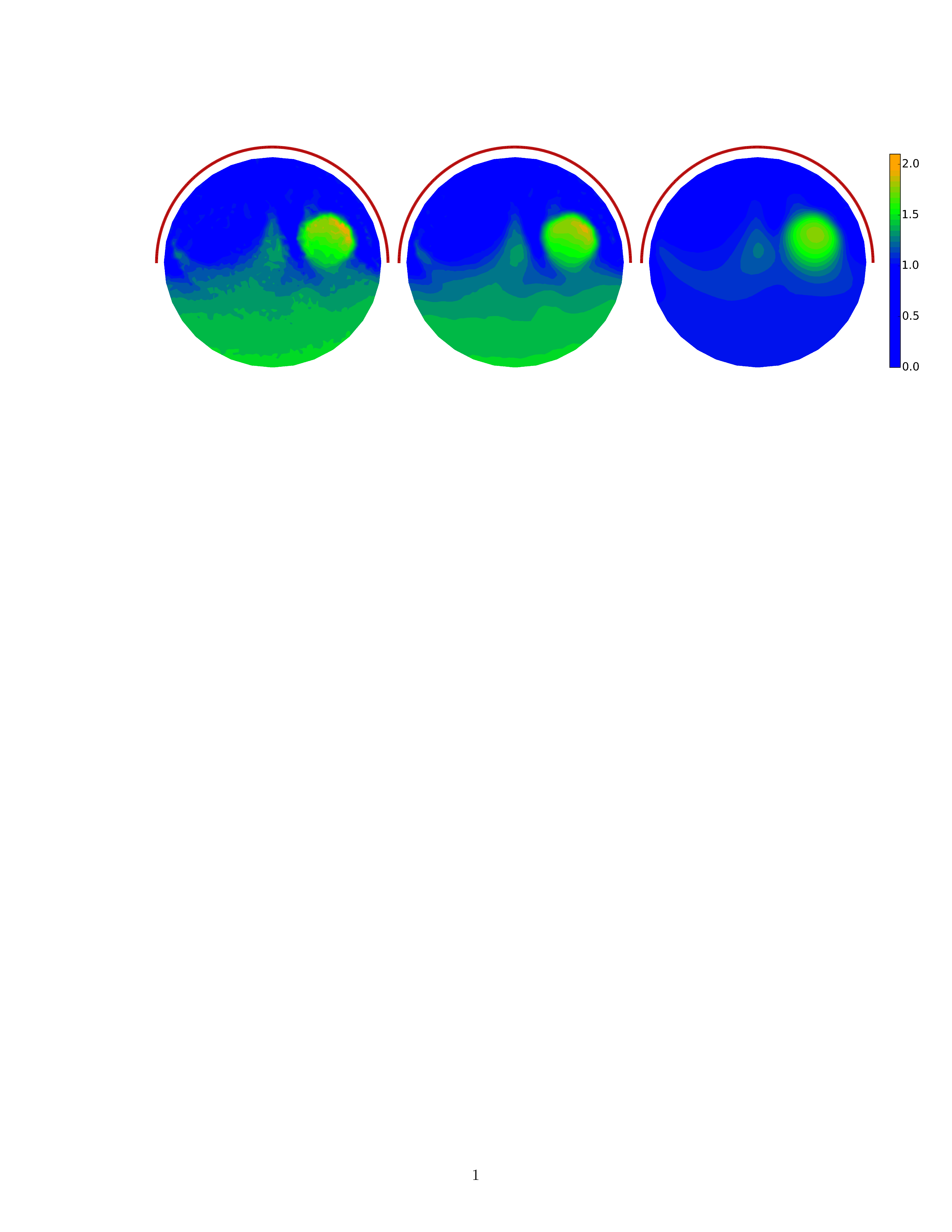}
		
		\includegraphics[width=0.9\textwidth, trim={3cm 19cm 0cm 3cm}, clip]{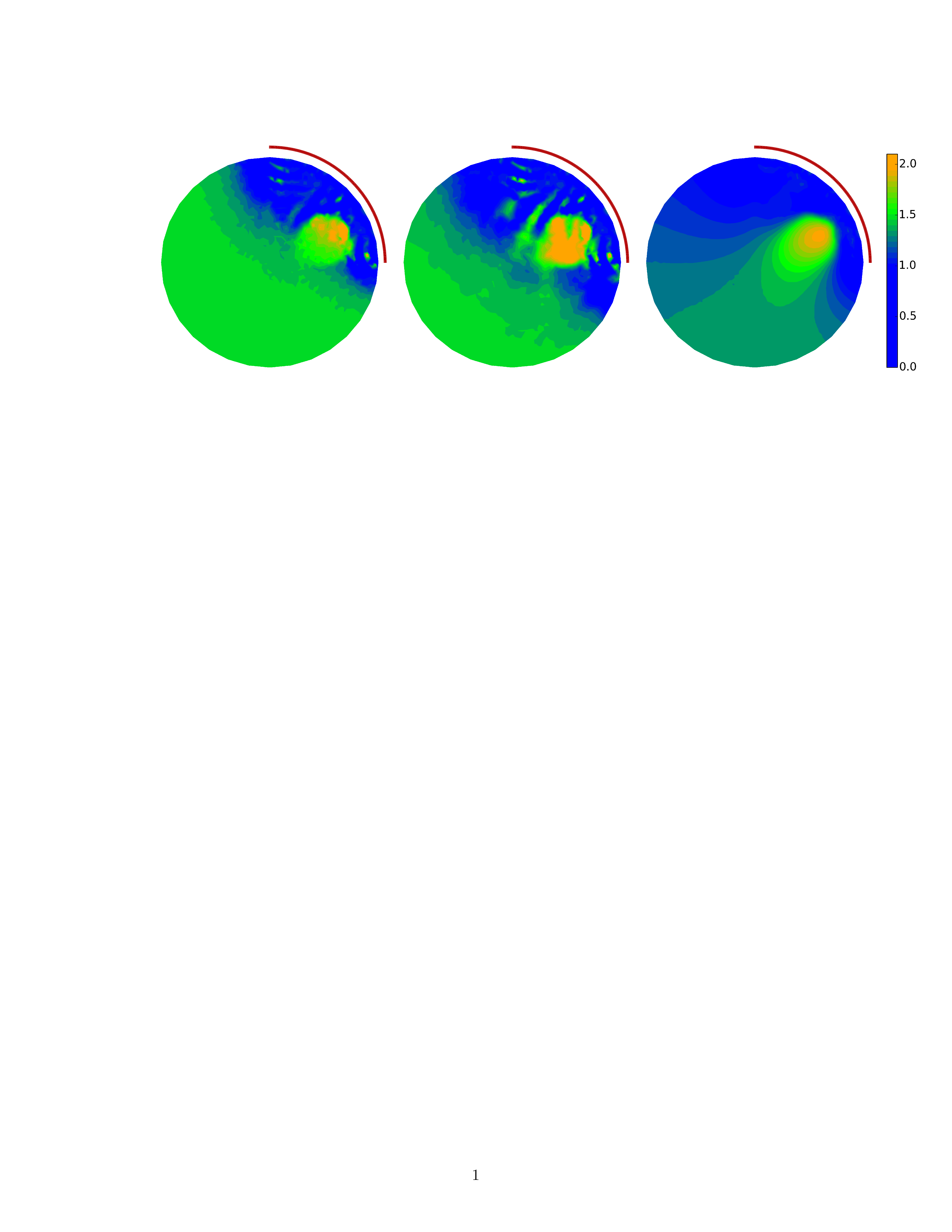}
		
		\caption{Reconstruction of conductivity $\sD$, Figure~\ref{fig_phantom}, from limited angle boundary conditions $g_j$, $j=1,2,3$. The red curves indicate the support of $g_j$. First column uses the $\Lt$ adjoint; second column, the $\Htb$ adjoint; and third  column, the $\Ht$ adjoint.}
		\label{fig_recon_100_75_50_25}
	\end{figure}

% % % % % % % % % % % % % % % % % % % % % % % %
% Results if the Ill-Posedness Quantification %
% % % % % % % % % % % % % % % % % % % % % % % %
\subsection{Results of the Ill-Posedness Quantification}\label{sec_IllPnumer}

In this section, we present some results from the ill-posedness quantification introduced in Section~\ref{sect_illp_quant} and show that the varying reconstruction results obtained for the different limited angle cases nicely correspond to certain pairs of singular values and vectors obtained from the SVD of $T$ defined by \eqref{transfer_matrix}.

First, we look at the condition numbers of $T$ for different limited angles and numbers of power density measurements, which are given in Table~\ref{tab-cond}. The transfer matrix $T$ becomes more and more ill-conditioned with decreasing angle and number of measurements, and therefore, we should not expect good reconstructions, especially further away from the accessible boundary. Additionally, we can see that using two measurements instead of one reduces the condition number of $T$ drastically, which should be compared with the identifiability results discussed in Section~\ref{sect_multi}. However, the third measurement does not reduce the condition number, but it remains of the same order, and therefore obtaining reasonable reconstructions with two measurements promises good reconstruction results as well, and with a shorter computational time.

	\begin{figure}[H] \centering
		\renewcommand\figurename{Table}
\resizebox{0.8\textwidth}{!}{%
		\begin{tabular}{| c | c | r | r | r | r | }
		  \hline			
		   Number of & Boundary & \multicolumn{4}{ c |}{Limited angle, \%} \\
		  \cline{3-6}			
		   measurements & functions  & 100    & 75     & 50       & 25	\\	
		  \hline	
		   3 & $g_1$, $g_2$, $g_3$ & 1.45 $\cdot 10^1$  & 3.77  $\cdot 10^2$  & 3.59  $\cdot 10^{3}$   & 8.81  $\cdot 10^{4}$   \\
		   \hline
		   2 & $g_1$, $g_2$ & 1.42 $\cdot 10^1$  & 3.70  $\cdot 10^2$  & 3.41  $\cdot 10^{3}$   & 8.09  $\cdot 10^{4}$   \\
		   2 & $g_2$, $g_3$ & 2.55 $\cdot 10^1$  & 9.55  $\cdot 10^2$  & 9.52  $\cdot 10^{3}$   & 2.14  $\cdot 10^{5}$   \\
		   2 & $g_1$, $g_3$ & 2.63 $\cdot 10^1$  & 3.76  $\cdot 10^2$  & 3.53  $\cdot 10^{3}$   & 8.44  $\cdot 10^{4}$   \\
		   \hline
		   1 & $g_1$        & 4.63 $\cdot 10^{3}$  & 6.43  $\cdot 10^{3}$  & 1.81  $\cdot 10^{5}$   & 4.52  $\cdot 10^{6}$\\
		   1 & $g_2$        & 2.15 $\cdot 10^{4}$  & 3.37  $\cdot 10^{5}$  & 5.31  $\cdot 10^{5}$   & 5.29  $\cdot 10^{5}$\\
		   1 & $g_3$        & 3.99 $\cdot 10^{3}$  & 5.98  $\cdot 10^{4}$  & 8.76  $\cdot 10^{4}$   & 7.20  $\cdot 10^{6}$\\
		  \hline 
		\end{tabular}
}		
		\caption{Condition numbers of the matrix $T$. Different combinations of boundary functions.}
		\label{tab-cond} 
	\end{figure}

In Figure~\ref{fig_singval_all}, the singular values for the different limited angle cases are plotted in descending order. One can see a decrease of the smallest singular values with the available angle, and as expected the problem becomes more ill-posed with less data.

Note that the last singular values seem to decay more rapidly. We believe that this is an effect of the numerical discretization and does not resemble the continuous problem. In light of this observation one could have truncated the singular values before computing the condition number, however the overall conclusion from Table \ref{tab-cond} would remain the same.

	\begin{figure}[H] \centering
		\includegraphics[width=0.5\textwidth]{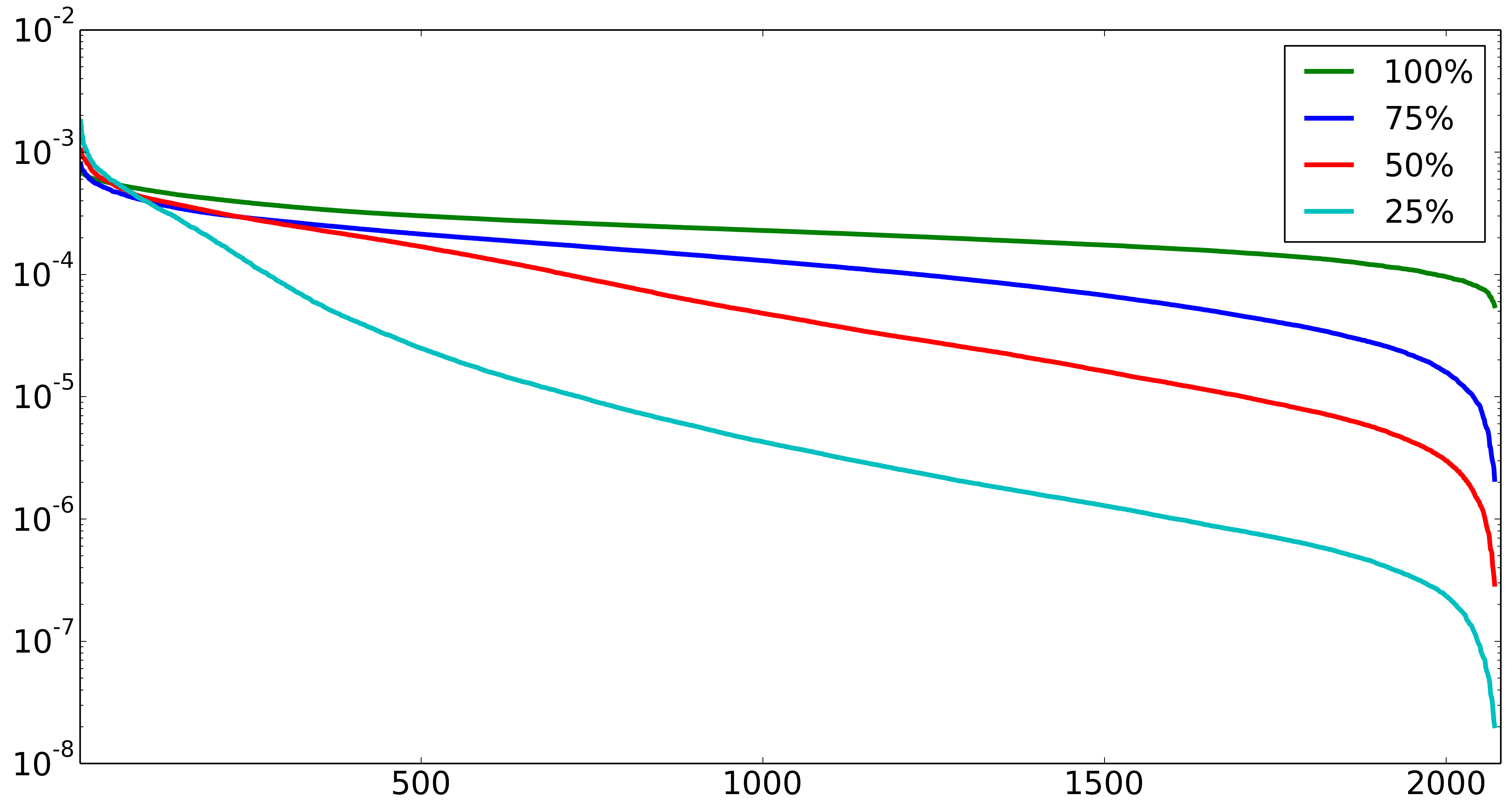}
		\caption{Singular values for $100 \%$, $75\%$, $50\%$ and $25\%$ available boundary with three measurements.}
		\label{fig_singval_all}
	\end{figure}
Moreover, in Figure~\ref{fig_singval_meas} we observe a similar decrease of the singular values depending on the number of measurements, thus confirming our conclusions about condition numbers.
 	\begin{figure}[H] \centering
		\includegraphics[width=0.5\textwidth]{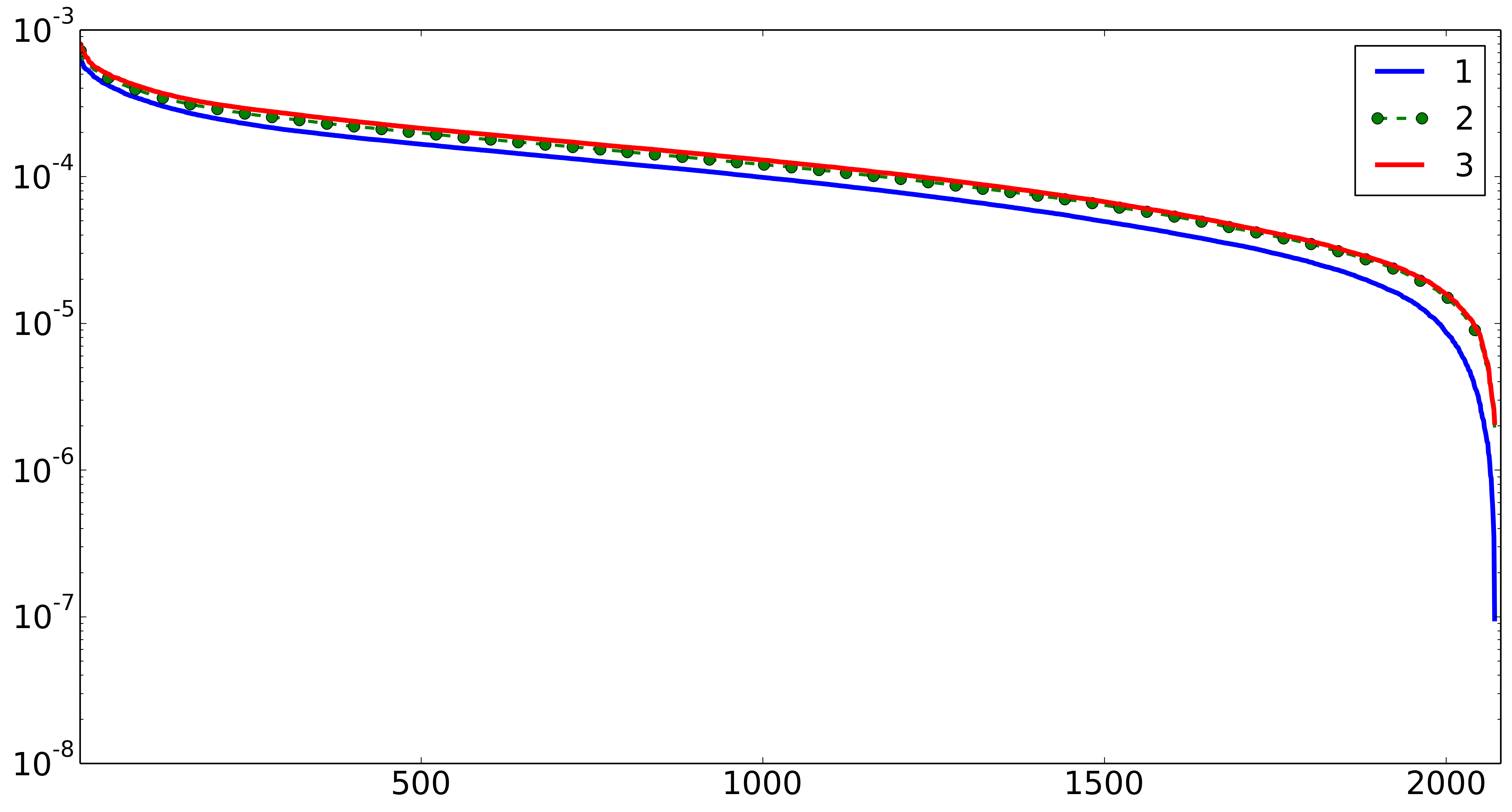}
		\caption{Singular values for $75 \%$ available boundary depending on the number of measurements.}
		\label{fig_singval_meas}
	\end{figure}

A selection of the resulting singular vectors for the Examples~\ref{ex_75}, \ref{ex_50} and \ref{ex_25} is depicted in Figures~\ref{fig_singvec_75}, \ref{fig_singvec_50} and \ref{fig_singvec_25}, respectively. The ordering of the singular values and singular vectors, denoted by $v_i$, is done in the usual way, i.e., the singular values are arranged in descending order, from the largest to the smallest, and the singular vector $v_1$ belongs to the largest singular value.
	
We see that different singular vectors carry information about the true conductivity $\sD$ in different areas of the domain. Unfortunately for the reconstruction, the singular vectors containing information about the area close to the inaccessible boundary correspond to small singular values. Since regularization methods have to rely on the singular vectors corresponding to larger singular values for a stable reconstruction, this adds to the explanation of the fact that close to the inaccessible boundary, the conductivity $\sD$ cannot be reconstructed. 
	
	\begin{figure}[H] \centering
		\includegraphics[width=0.32\textwidth, trim={3cm 19cm 11.3cm 3cm}, clip=true]{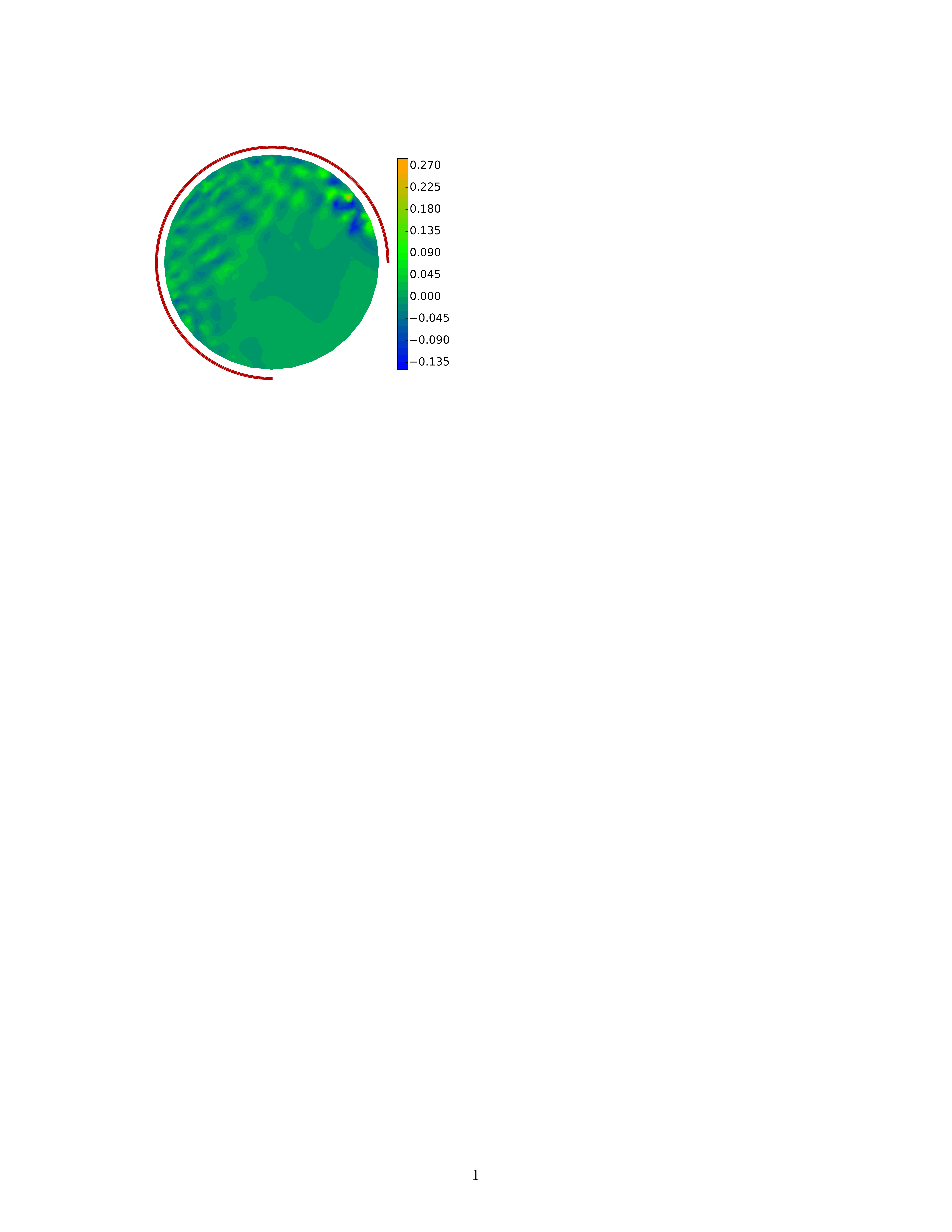}
		\includegraphics[width=0.32\textwidth, trim={3cm 19cm 11.3cm 3cm}, clip=true]{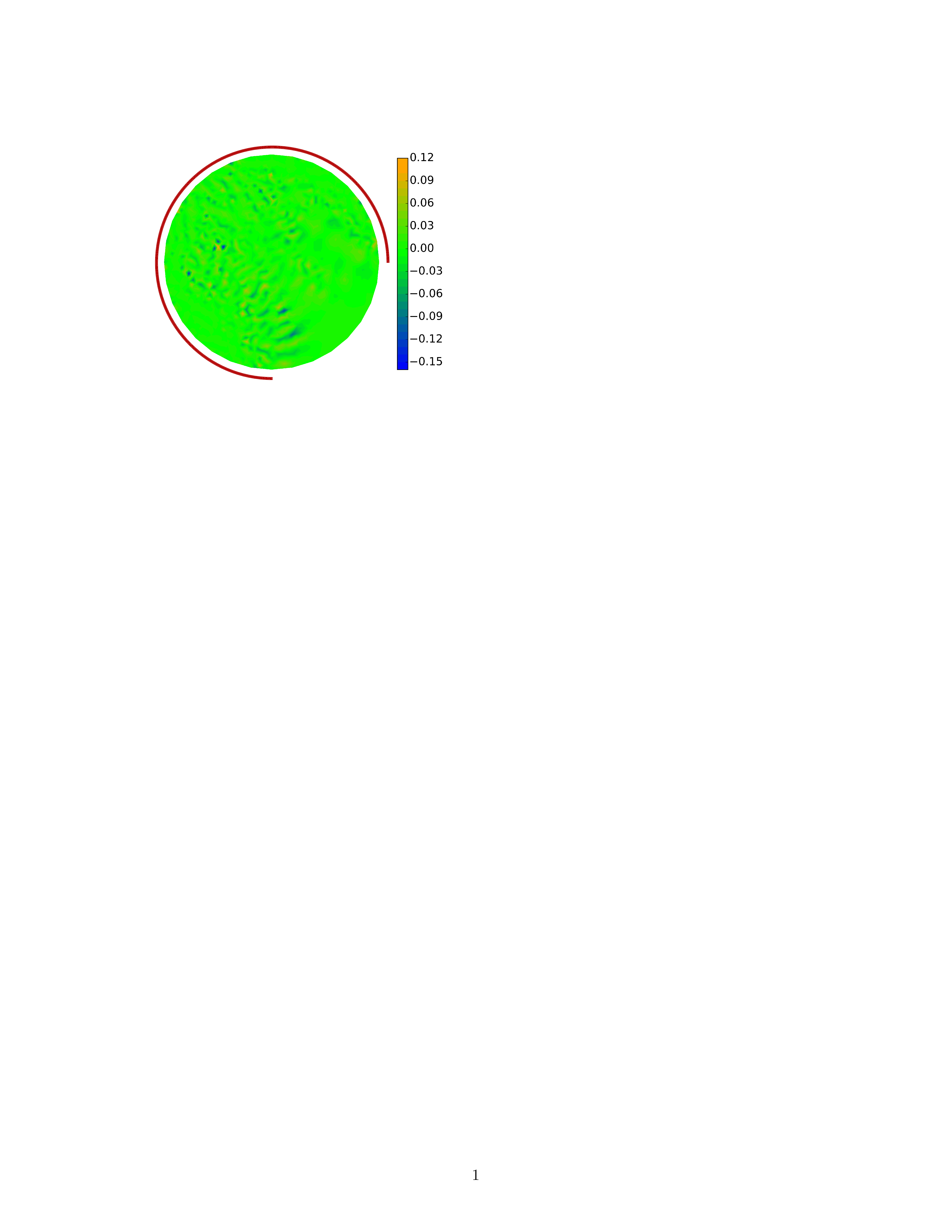}
		\includegraphics[width=0.32\textwidth, trim={3cm 19cm 11.3cm 3cm}, clip=true]{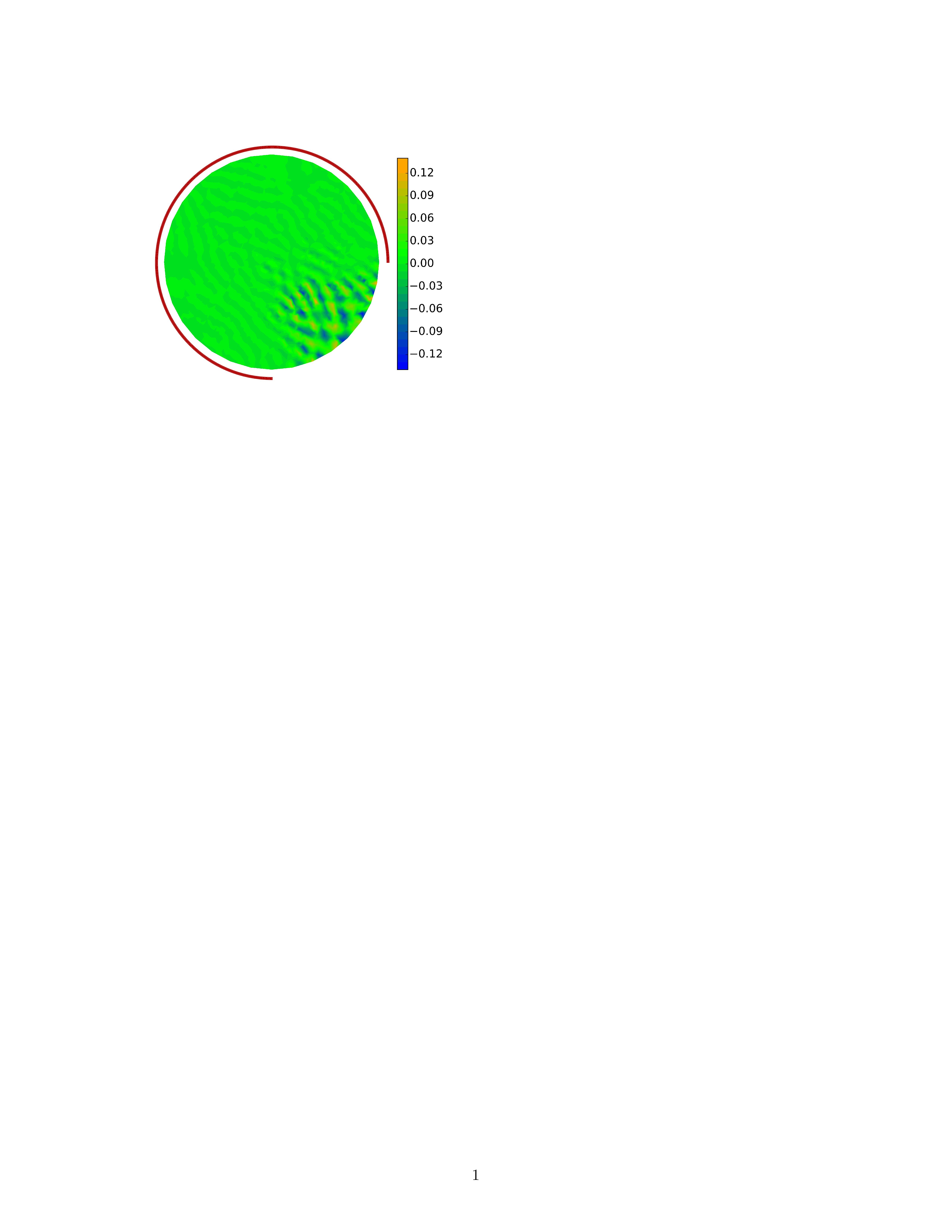}
		\caption{
		Singular vectors of the matrix $T$ from boundary conditions $g_j$, $j=1,2,3$ defined in \eqref{bdc_theta} with $\alpha=2\pi/3$. From left to right: $v_{100}$, $v_{1000}$, $v_{2060}$.}
		\label{fig_singvec_75}
	\end{figure}

	\begin{figure}[H] \centering
		\includegraphics[width=0.32\textwidth, trim={3cm 19cm 11.3cm 3cm}, clip=true]{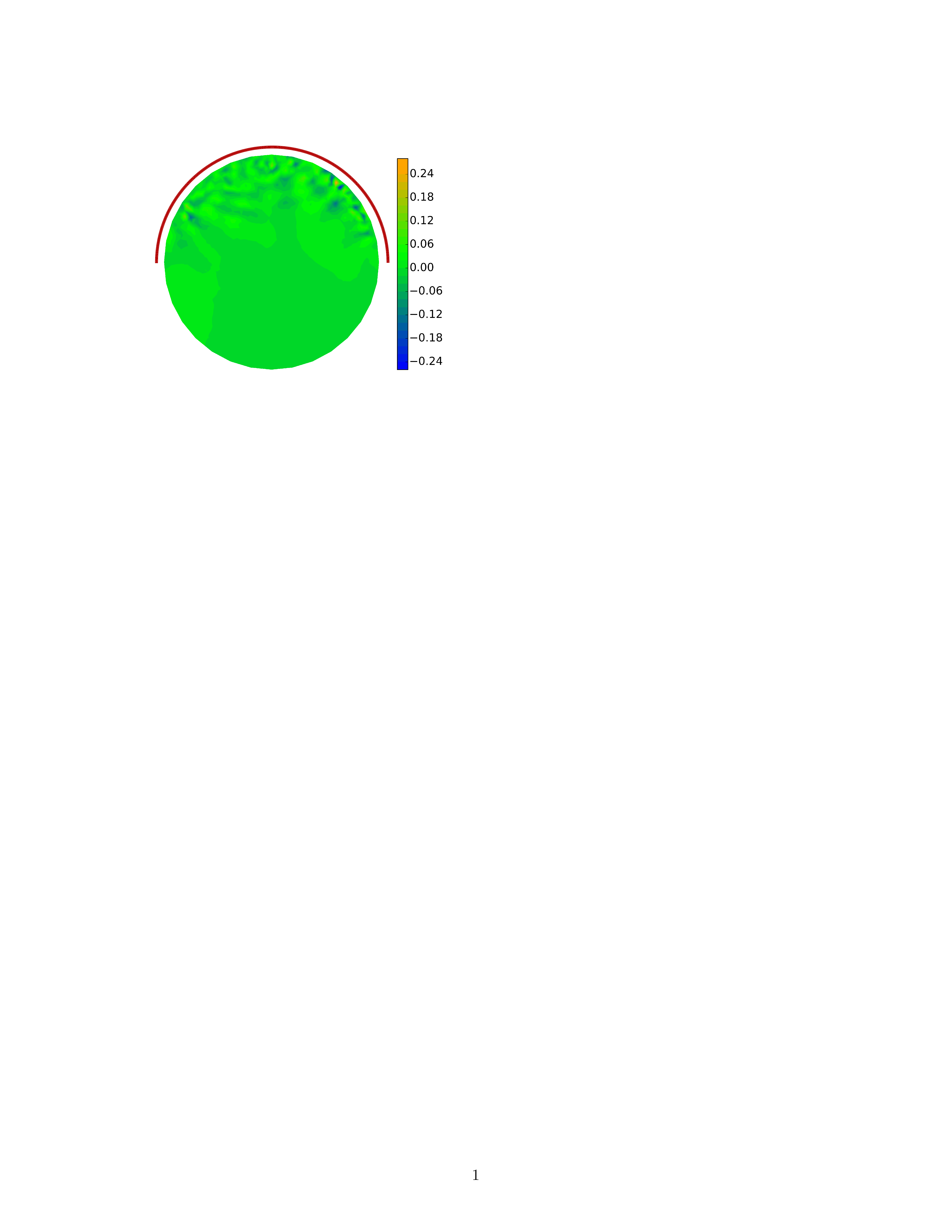}
		\includegraphics[width=0.32\textwidth, trim={3cm 19cm 11.3cm 3cm}, clip=true]{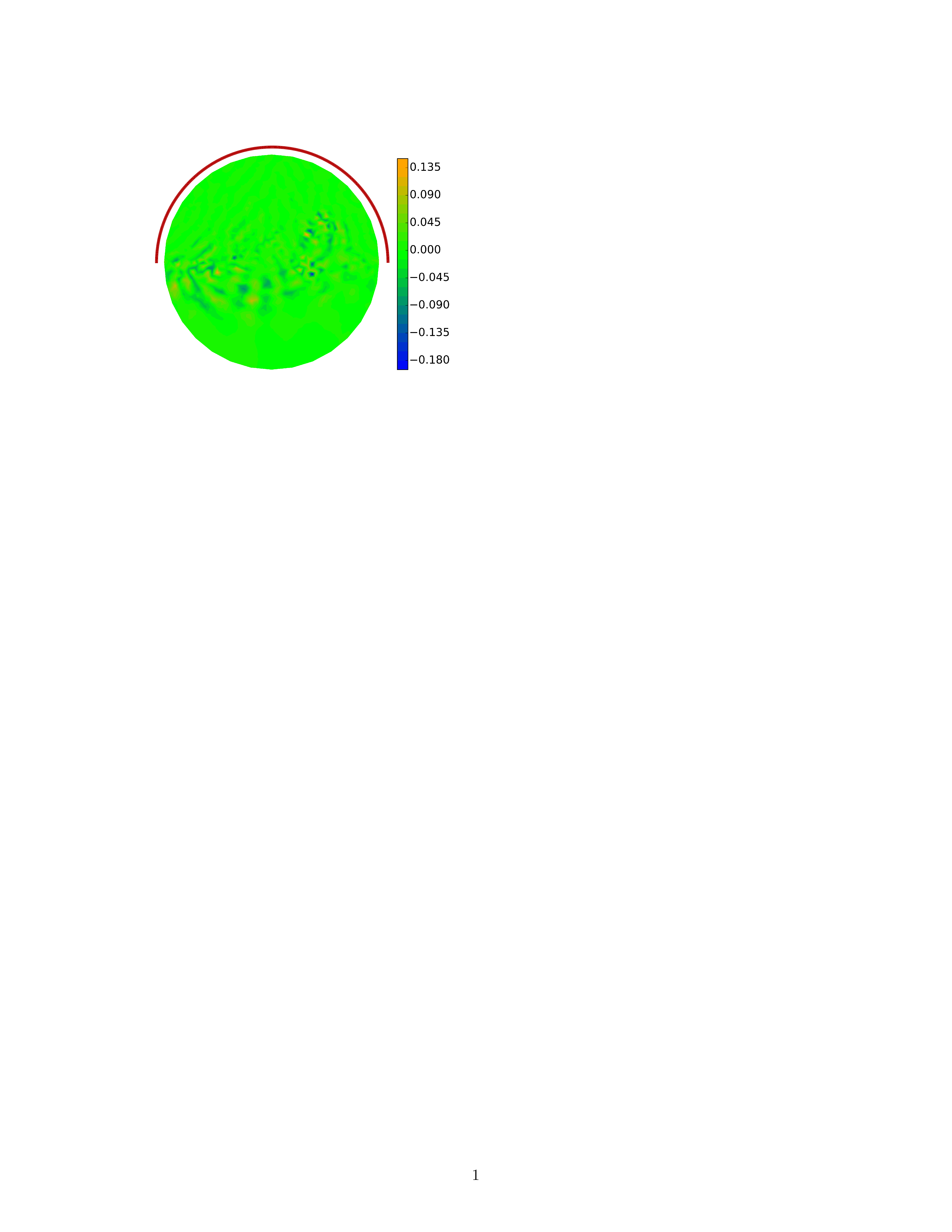}
		\includegraphics[width=0.32\textwidth, trim={3cm 19cm 11.3cm 3cm}, clip=true]{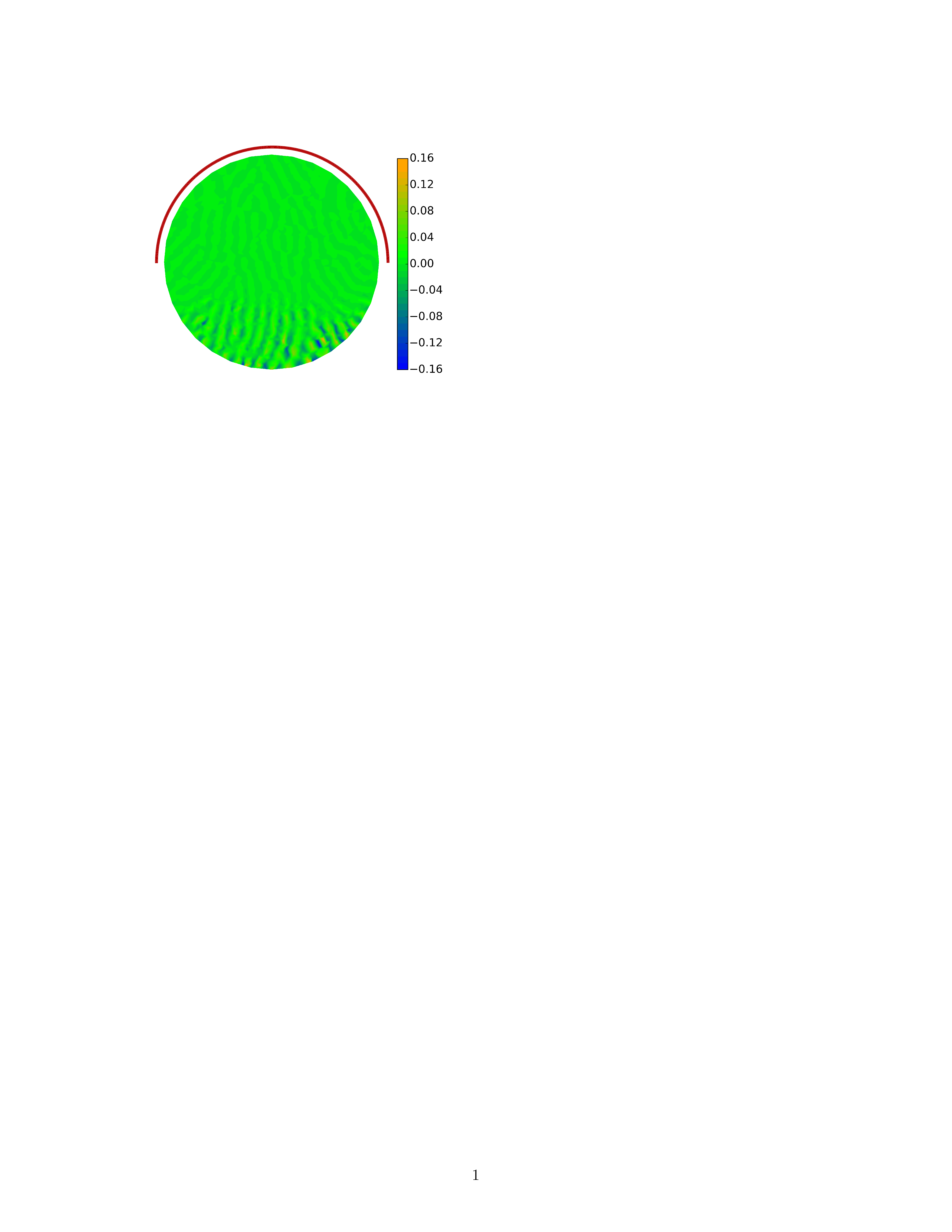}
		\caption{
		Singular vectors of the matrix $T$ from boundary conditions $g_j$, $j=1,2,3$ defined in \eqref{bdc_theta} with $\alpha=\pi$. From left to right: $v_{100}$, $v_{1000}$, $v_{2060}$.}
		\label{fig_singvec_50}
	\end{figure}	

	\begin{figure}[H] \centering
		\includegraphics[width=0.32\textwidth, trim={3cm 19cm 11.3cm 3cm}, clip=true]{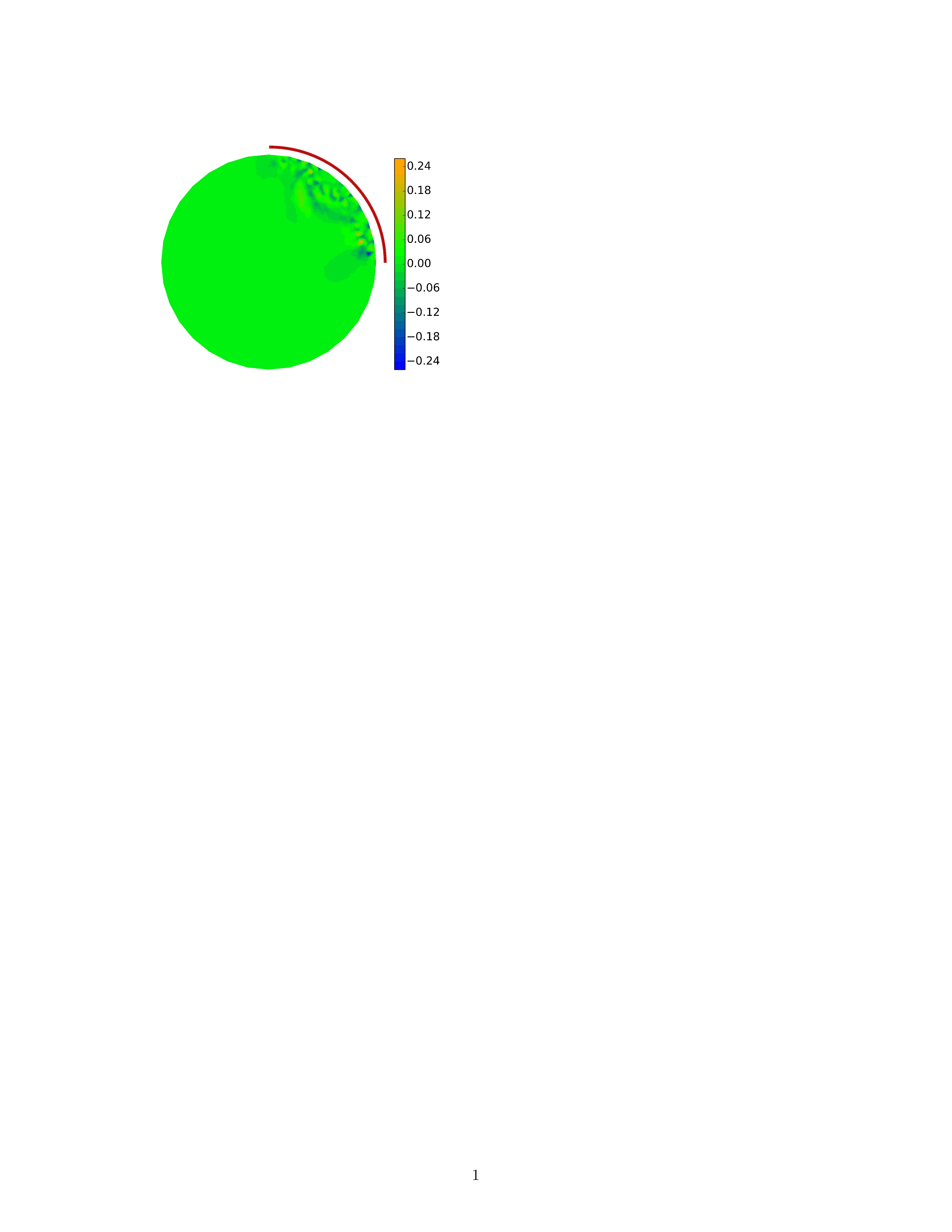}
		\includegraphics[width=0.32\textwidth, trim={3cm 19cm 11.3cm 3cm}, clip=true]{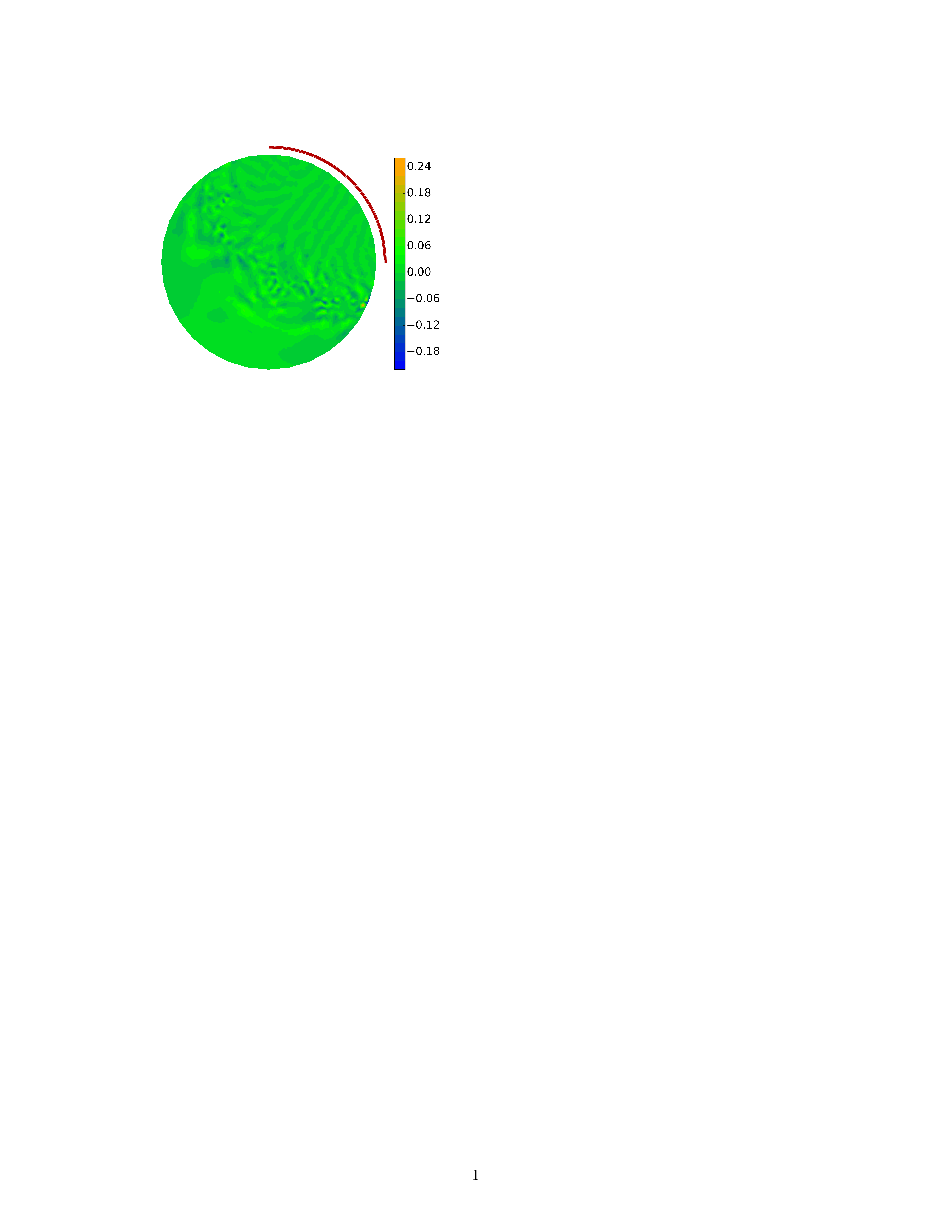}
		\includegraphics[width=0.32\textwidth, trim={3cm 19cm 11.3cm 3cm}, clip=true]{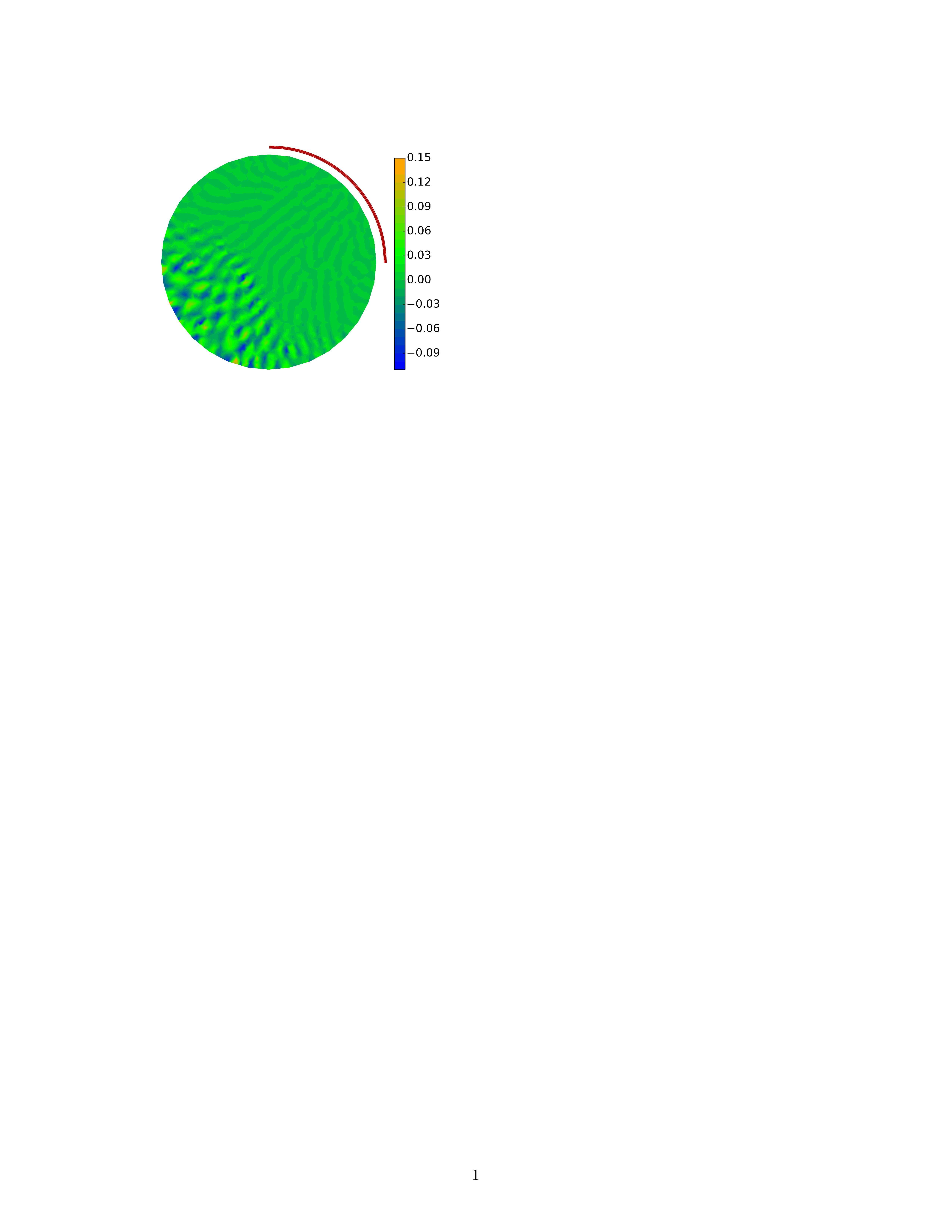}
		\caption{
		Singular vectors of the matrix $T$ from boundary conditions $g_j$, $j=1,2,3$ defined in \eqref{bdc_theta} with $\alpha=\pi/2$. From left to right: $v_{100}$, $v_{1000}$, $v_{2060}$.}
		\label{fig_singvec_25}
	\end{figure}

% % % % % % % %
% Conclusions %
% % % % % % % %
\section{Conclusions}
We formulated the hybrid imaging problem of estimating a spatially varying conductivity $\sigma$ from measurements of power densities resulting from different prescribed boundary currents in an infinite dimensional setting and  presented various numerical results, focusing especially on the limited angle case. In particular, we saw that reconstructing the conductivity is difficult far away from the accessible part of the boundary, due to lack of information in this area. Through a numerical ill-posedness quantification, we were able to establish a close connection between the reconstruction quality and the SVD of the Fr\'echet derivative of $F$.  As the size of the accessible boundary becomes smaller, the reconstruction quality deteriorates, which is confirmed by a rapid decay of the corresponding singular values. The degree of ill-posedness of the linearized problem decreases with the size of the accesible boundary and if more than one measurement is used. The obtained results shed some light on the influence of limited angle data in hybrid tomography, clearly illustrating the possibilities and limitations in numerical practise. Other measures of ill-posedness quantification than condition numbers can be suggested, such as the decay rate of the singular values.

% % % % % % % % % % % % % % % % %
% Support and Acknowledgements  %
% % % % % % % % % % % % % % % % %
\section{Support and Acknowledgements}
ES was funded by the Danish Council for Independent Research $|$ Natural Sciences: grant 4002-00123. SH was funded by the Austrian Science Fund (FWF): W1214-N15, project DK8. The authors would like to acknowledge the M.Sc.\ work of Christina Hilderbrandt, which includes early ideas on ill-posedness quantification for the limited angle problem with Dirichlet boundary conditions.

% % % % % % % % %
% Bibliography  %
% % % % % % % % %

\bibliographystyle{plain}
\bibliography{mybib}

\end{document}